%% file: Transmision.tex
\documentclass[11pt]{article}
\input{ex_shared}
\usepackage{a4wide}
\usepackage{authblk}

\begin{document}

\title{Fast Solver for Quasi-Periodic 2D-Helmholtz Scattering in Layered Media\thanks{This work was partially funded by Fondecyt Regular 1171491 and by grants Conicyt-PFCHA/Doctorado Nacional/2017-21171791 and 2017-21171479.}}
\author[1]{Jos\'e Pinto}
\author[2]{Rub\'en Aylwin}
\author[3]{Carlos Jerez-Hanckes}
 %\authormark{1,*}
%Gerardo Silva-Oelker,\\%\authormark{2}
%Carlos Jerez-Hanckes %\authormark{3}
 %and Patrick Fay%\authormark{4}
%}
{\tiny 
\affil[1]{Department of Electrical Engineering, Pontificia Universidad Cat\'olica de Chile, Santiago, Chile.}
\affil[2]{Department of Electrical Engineering, Pontificia Universidad Cat\'olica de Chile, Santiago, Chile.}
\affil[3]{Faculty of Engineering and Sciences, Universidad Adolfo Ib\'a\~nez, Santiago, Chile.}
}

%\thanks{This work was partially funded by Fondecyt Regular 1171491 and by grants Conicyt-PFCHA/Doctorado Nacional/2017-21171791 and 2017-21171479.}
%
%\author{Jos\'e Pinto}\address{School of Engineering, Pontificia Universidad Cat\'olica de Chile, Santiago, Chile (\email{jspinto@uc.cl\ \&\ rdaylwin@uc.cl}).}
%\author{Ruben Aylwin}\sameaddress{1}
%\author{Carlos Jerez-Hanckes}\address{Faculty of Engineering and Sciences, Universidad Adolfo Ib\'a\~nez, Santiago, Chile (\email{carlos.jerez@uai.cl}).}
%\subjclass{65N35, 65N38, 45M15, 78A45}
%\keywords{Boundary Integral Equations, Quasi-periodic scattering, Spectral Elements, Gratings, Multi-layered domain}
\maketitle

\begin{abstract}
We present a fast spectral Galerkin scheme for the discretization of
boundary integral equations arising from two-dimensional Helmholtz transmission problems in
multi-layered periodic structures or gratings. Employing suitably parametrized Fourier basis and excluding Rayleigh-Wood anomalies,
we rigorously establish the well-posedness of both continuous and discrete
problems, and prove super-algebraic error convergence rates for the proposed scheme. Through several numerical examples, we confirm our findings and show performances competitive to those attained via Nystr\"om methods. 
\end{abstract}

\section{Introduction}
A vast number of scientific and engineering applications rely on harnessing acoustic and electromagnetic wave diffraction by periodic and/or multilayered domains. Current highly demanding operation conditions for such devices require solving thousands of specific settings for design optimization or the quantification of shape or parameter uncertainties in the relevant quantities of interest, challenging the scientific computing community to continuously develop ever more efficient, fast and robust solvers (cf.~\cite{Bao:2004,chen2007design,loewen2018diffraction,silva2018quantifying, SJF19} and references therein). Assuming impinging time-harmonic plane waves, scattered and transmitted fields have been solved by a myriad of mathematical formulations and associated solution schemes. These range from volume variational formulations to various boundary integral representations and equations (cf.~\cite{ammari1998scattering,ammari2001analysis,
Bao:2000,barnett2011new,dobson1992time,nakata1990boundary}), pure or coupled implementations of finite and boundary element methods (cf.~\cite{Ammari:2008,ammari2001analysis, Elschner:1998,Nedelec:1991,silva2018quantifying} or \cite[Chapter 5]{popov2012gratings}) and Nystr\"om methods \cite{bruno2009efficient,bruno2016superalgebraically,Barnett:2015,greengard2014fast,liu2016efficient}.

In this work, we build upon our theoretical review given in \cite{APJ:2019} and present a spectral Galerkin method for solving second-kind direct boundary integral equations (BIEs) for the Helmholtz transmission problem for two-dimensional, periodic multi-layered
gratings with smooth interfaces. Contrary to the low-order local basis functions used in the standard boundary element method, spectral bases are composed of high-order polynomials whose support lie on the whole scatterer boundary or on large portions of it. Successfully employed on two- and three-dimensional scattering problems \cite{PaperB,HU1995340,graham2002fully}, the main advantage of a spectral discretization is the ability to converge at a super-algebraic rate whenever solutions are smooth enough. Hence, our proposed method can in practice
compete with Nystr\"om methods while simultaneously inheriting
all of the theoretical aspects of classical Galerkin methods.

In two dimensions, spectral
methods are closely related to the theory of periodic pseudo-differential operators \cite{saranen2013periodic}, since the discretization
through spectral elements can be interpreted as a truncation of the
associated Fourier series where the action of the
operators is well understood. We show that wave scattering by periodic domains is closely connected to the bounded domain case, making it possible
to reuse almost all the pseudo-differential operator theory for our analysis. 
Key to our analysis are the
results in \cite{Nedelec:1991,Starling:1994,Elschner:1998} regarding the unique solvability
and eigenvalues of the associated volume problem. From here, we deduce
that our BIE is uniquely solvable except at a countable set of wavenumbers composed of Rayleigh-Wood
frequencies---wavenumbers for which the sum defining the quasi-periodic
Green's function is not convergent---and of eigenvalues of the
Helmholtz transmission problem. Mindless of the several remedies
developed to tackle Rayleigh-Wood anomalies through BIEs
\cite{Barnett:2015,bruno2014rapidly,bruno2017three,bruno2017rapidly}, we
choose to avoid them as they are not captured by our previous analysis in \cite{APJ:2019}.

Our discretization method employs a quasi-periodic basis so that techniques forcing the quasi-periodicity of the discrete solutions are not necessary
(cf.~\cite{greengard2014fast,zhang2019fast}). Instead, an accurate approximation of the quasi-periodic Green's function is required in order to extract its Fourier coefficients
through the fast Fourier transform (FFT). Moreover, we prove that the chosen discretization basis enjoys a super-algebraic convergence rate on the degrees of freedom, which we then confirm through numerical experiments. In \cite{Nguyen:2012}, a similar quasi-periodic exponential basis was employed to approximate solutions of a volume integral formulation.

The article is structured as follows. Section \ref{sec:notation} presents
the notation used throughout as well as the required quasi-periodic Sobolev spaces setting following \cite{APJ:2019}. In Section \ref{sec:wavepropagationvolume} we state the Helmholtz transmission problem for a multi-layered grating and study its solvability. Section \ref{sec:BIEBIO} is concerned with
the properties of quasi-periodic boundary
integral operators (BIOs) along with an existence and uniqueness
result for our BIEs. Section \ref{sec:specmethod} provides rigorous
error convergence rates of the spectral method and briefly describes the
numerical algorithm used to compute the matrix entries associated with
each integral operator. Numerical results are discussed in
Section \ref{sec:numexam}, followed by concluding remarks on Section \ref{sec:conclusions}.

%%%%%%%%%%%%%%%%%%%%%%%%%%%%%%%%%%%%
\section{Notation and Functional Space Setting}
\label{sec:notation}
%%%%%%%%%%%%%%%%%%%%%%%%%%%%%%%%%%%%
\subsection{General Notation}
We denote the imaginary unit $\imath$. Boldface symbols will denote vectorial quantities and will use greek and roman letters for data over boundaries and volume, respectively. Canonical vectors in $\IR^2$ are denoted $\bm{e_1},\bm{e_2}$ respectively. Also, we make use of the symbols $\lesssim$, $\gtrsim$ and $\cong$ to avoid specifying constants irrelevant for the corresponding analysis.

Let $H$ be a given Banach space. We shall denote its norm
as $\norm{\cdot}{H}$ and its dual space by $H'$ (set of antilinear functionals over $H$) with dual product denoted by $\p{\cdot,\cdot}$.
If $H$ is a Hilbert space, the inner product between two
of its elements, $x$ and $y$, is denoted as $\pr{x,y}_H$.
Moreover, if $H$ is a Hilbert space over the complex field, the inner
product will be understood in the anti-linear sense.

For an open domain $\Omega{\subset}\IR^2$, its boundary shall be
denoted as $\partial\Omega$. Moreover, for any $\cO{\subset}\IR^2$ such
that $\Omega\subseteq\cO$, we introduce the closure of $\Omega$
relative to $\cO$ as $\overline\Omega^{\cO}:=\overline\Omega\cap\cO$
and the boundary of $\Omega$ relative to $\cO$ as
$\partial^{\cO}\Omega:=\overline{\Omega}^{\cO}\setminus\Omega$.

For $n\in{\IN}_0:=\IN\cup\{0\}$, we denote by $\C^n(\Omega)$ the set of scalar
functions over $\Omega$ with complex values and continuous derivatives
up to order $n$. $\C^\infty(\Omega)$ refers to the space
of functions with infinite continuous derivatives over $\Omega$.
We shall also make use of the following subset of $\C^{\infty}(\Omega)$:
$$\cD(\Omega):=\{u\in\C^\infty(\Omega)\ :\ \text{supp }u\subset\subset \Omega\}.$$
The space of $p$-integrable functions (for $p\geq 1$) with complex values over $\Omega$
is denoted as $L^p(\Omega)$. 

We say that a one-dimensional Jordan curve $\Gamma$ is of class
 $\C^{r,1}$, for $r\in\IN_0$, if it may be parametrized by a function
$\bm{z}:(0,2\pi)\to\Gamma$ which has $r$ Lipschitz-continuous
derivatives and a non-vanishing tangential vector. The first derivative of the parametrization is denoted as $\dot{\bm{z}}$.
 Moreover, we say
$\Gamma$ is of class $\C^{\infty}$ if it is of class $C^{r,1}$ for every $r \in \IN_0$ (we will also use the notation $\mathcal{C}^{\infty,1}$ to refer to the same class).

Throughout the following sections, we will consider periodic
geometries along $\bm{e_1}$ with a fixed period of
$2 \pi$. Moreover, we say that a continuous function $f$
is a $\theta$-quasi-periodic function if, 
\begin{align*}
f(\bx+2\pi \bm{e}_1)=e^{\imath 2\pi\theta}f(\bx)\quad\forall\bx\in\IR^2,
\end{align*} 
where the quasi-periodic shift $\theta$ is always assumed to be in $[0,1)$.
Finally, we define the canonic periodic cell on $\IR^2$ as $\mathcal{G} := (0,2\pi) \times \IR$. 

\subsection{Quasi-periodic Sobolev Spaces}

We denote by $\cD_\theta(\IR^2)$ the space
of $\theta$-quasi-periodic functions in $\C^\infty(\IR^2)$
that vanish for large $\vert x_2\vert$, and denote by
$\cD_\theta'(\IR^2)$ the space of $\theta$-quasi-periodic distributions,
which can be seen as the dual space of $\cD_\theta(\IR^2)$
(cf.~\cite[Proposition 2.4]{APJ:2019}).
For $\cG$ as before, we introduce $\cD_\theta(\cG)$
the space of restrictions to $\cG$ of
elements in $\cD_\theta(\IR^2)$. Moreover, 
for any open domain $\Omega\subset\cG$
we define $\cD_\theta(\Omega)$ as the set of elements of
$\cD_\theta(\cG)$ with compact support on $\Omega$ and
$\cD'_\theta(\Omega)$ as the space of
elements of $\cD'_\theta(\cG)$ restricted to $\cD_\theta(\Omega)$. In what follows, for all $j\in\IZ$ we define $j_\theta:=j+\theta$. 

\begin{prpstn}[Proposition 2.6 in \cite{APJ:2019}]
\label{prop:FourierSeries}
Every $u\in\cD_\theta(\IR^2)$ can
be represented as a Fourier series, i.e.
\begin{align*}
u(\bx)=\sum\limits_{j\in\IZ}u_j(x_2)e^{\imath j_\theta x_1}\quad \text{with}\quad u_j(x_2):=\frac{1}{2\pi}\int_{0}^{2\pi}e^{-\imath j_\theta x_1}u(\bx) \ \d\!x_1,
\end{align*}
so that $u_j\in\cD(\IR)$. On the other hand, every element $F \in  \cD'_\theta(\IR^2)$ can be identified with a formal Fourier series given by 
\begin{align*}
\sum\limits_{j\in\IZ}F_je^{\imath j_\theta x_1},\quad \text{with}\quad {F}_j:=\begin{cases}\cD(\IR)&\to\IC\\
v&\mapsto F(v(x_2)e^{\imath j_\theta x_1})
\end{cases},
\end{align*} 
where  $F_j\in\cD'(\IR)$ for all $j\in\IZ$ and $F(u)=\sum\limits_{j\in\IZ}F_j(u_j)$.
\end{prpstn}

Let $s\in\IR$. We define the  $\theta$-quasi-periodic Sobolev
space of order $s$ on $\cG$ as follows,
\begin{align*}
H^s_\theta(\cG):=\left\lbrace F\in\cD'_\theta(\IR^2)\ \bigg|\ \sum\limits_{j\in\IZ}\int_{\IR}(1+j_\theta^2+\modulo{\xi}^2)^{s}\modulo{\widehat{F}_j(\xi)}^2\ \d\!\xi<\infty\right\rbrace,
\end{align*}
wherein $\widehat{F}_j$ is the Fourier transform (in distributional sense \cite[Section 2.4]{Steinbach:2007aa}) of $F_j$, defined as in Proposition \ref{prop:FourierSeries}.
Additionally, we introduce the common notation $L^2_\theta(\cG):=H^0_\theta(\cG)$
and note that, as in the standard case, $H^s_\theta(\cG)$ is a Hilbert space
\cite[Proposition 2.8]{APJ:2019}. Furthermore, for an open proper subset
$\Omega$ of $\cG$, we define $H^s_\theta(\Omega)$ as the Hilbert space of
restrictions to $\Omega$ of elements of $H^s_\theta(\cG)$ (see \cite[Section 2]{APJ:2019} and \cite[Chapter 3.6]{McLean:2000}).
Finally, local Sobolev spaces on $\Omega$ are defined as
$$H^s_{\theta,\text{loc}}(\Omega):=\left\lbrace{u\in\cD'_\theta(\Omega)\ :\ u\in H^s_\theta\left(\Omega\cap\left\{\bx\in\cG\; :\;\modulo{x_2}<R\right\}\right)\quad \forall\ R>0}\right\rbrace.$$
\subsection{Quasi-periodic Sobolev Spaces on Boundaries and Traces}
We begin by considering spaces of periodic functions over $\IR$. As in \cite[Definition 8.1]{kress:2014}, 
\cite[Section 5.3]{saranen2013periodic}, we define Sobolev spaces on
$[0,2\pi]$ of order $s\geq 0$ as follows,
\begin{align*}
H^s[0,2\pi]:=\left\lbrace \phi\in L^2((0,2\pi))\ :\ \sum\limits_{j\in\IZ}(1+j^2)^s\modulo{\phi_j}^2<\infty \right\rbrace,
\end{align*}
where $\{\phi_j\}_{j\in\IZ}$ are the Fourier coefficients of $\phi$. Quasi-periodic spaces of order $s\geq 0$ over $(0,2\pi)$
are defined from $H^s[0,2\pi]$ straightforwardly, i.e.,
\begin{align*}
H^s_\theta[0,2\pi]:=\left\lbrace \phi\in L^2((0,2\pi))\ :\ e^{-\imath \theta t}\phi(t)\in H^s[0,2\pi] \right\rbrace.
\end{align*}
Both $H^s[0,2\pi]$ and $H^s_\theta[0,2\pi]$ are Hilbert spaces, as are their respective dual spaces, denoted respectively $H^{-s}[0,2\pi]$ and $H^{-s}_\theta[0,2\pi]$ (see \cite[Theorem 8.10]{kress:2014}
and \cite[Theorem 2.20]{APJ:2019}). Moreover, for $s\in\IR$, the inner
product and norm of $H^s_\theta[0,2\pi]$ are given by:
\begin{align*}
\pr{u,v}_{H^s_\theta[0,2\pi]}:=\sum\limits_{j\in\IZ}(1+j_\theta^2)^su_{j,\theta}\overline{v_{j,\theta}}\quad\text{and}\quad\norm{u}{H^s_\theta[0,2\pi]}:=\pr{u,u}_{H^s_\theta[0,2\pi]}^{\half},
\end{align*}
wherein, for positive $s$, we define
\begin{align*}
u_{j,\theta}:=\frac{1}{2\pi}\pr{u(t),e^{\imath j_\theta t}}_{L^2((0,2\pi))},
\end{align*}
and the product is extended through duality to negative $s$ (cf.~\cite[Theorems 2.16 and 2.20]{APJ:2019}).

We continue by considering boundaries which are constructed as the single period of a periodic Jordan curve of class $\mathcal{}C^\infty$. Let $\Gamma$ be one of such curves and let $\bm{z} :(0,2\pi) \rightarrow \Gamma$ be a parametrization of $\Gamma$.  Then, for any $s\geq 0$, we define the
$\theta$-quasi-periodic Sobolev space of order $s$ on $\Gamma$ as 
\begin{align*}
H^s_\theta(\Gamma):=\left\lbrace u\in L^2_\theta(\Gamma)\ \vert\ (u\circ\bz)(t)\in H_\theta^s[0,2\pi] \right\rbrace.
\end{align*}
We define $H^{-s}_\theta(\Gamma)$ as the completion of
$L^2_\theta(\Gamma)$ under the norm
given by
\begin{align*}
\norm{u}{H^{-s}_\theta(\Gamma)}:=\mathlarger{\Vert}(u\circ\bz)\norm{\dot{\bz}}{\IR^2}\mathlarger{\Vert}_{H^{-s}_\theta[0,2\pi]}.
\end{align*}
Norms and inner products for these spaces are
given through their respective pullbacks to $H^s_\theta[0,2\pi]$ and
$H^{-s}_\theta[0,2\pi]$. Moreover, $H^{-s}_\theta(\Gamma)$ is identified with the
dual space of $H^{s}_\theta(\Gamma)$ \cite[Theorem 2.26]{APJ:2019} where the duality is given by the extension of the following anti-liner form: 
\begin{align}
\label{eq:dualprod}
\p{ \lambda , \vartheta}_\Gamma := \pr{\lambda,\vartheta}_{L^2_\theta (\Gamma)}, \quad \lambda, \vartheta \in L^2_\theta (\Gamma).
\end{align}
We also define the following space of smooth functions over $\Gamma$,
\begin{align*}
\cD_\theta(\Gamma):=\left\lbrace \phi:\Gamma\to\IC\ {\bigg|}\ (\phi\circ\bz)(t)=\sum\limits_{j=-n}^{n}\phi_je^{\imath j_\theta t},\text{ for some }n\in\IN\right\rbrace,
\end{align*}
which is dense on $H^s_\theta(\Gamma)$ for any $s\in\IR$.
Finally, we introduce trace operators acting on quasi-periodic
Sobolev spaces. Let $\Omega$ be a proper open subset of $\cG$ such that $\partial^{\cG}\Omega=\Gamma$,
we define the following operators for $s>\half$:
\begin{align*}
\tD :H^s_{\theta}(\Omega)\to H^{s-\half}_\theta(\Gamma),\quad \tD^{e} :H^s_{\theta}(\cG\setminus\overline\Omega^{\cG})\to H^{s-\half}_\theta(\Gamma),
\end{align*}
that extend the notion of the restriction operator
$u\mapsto u|_{\Gamma}$ to quasi-periodic Sobolev spaces
\cite[Theorem 2.29]{APJ:2019}.
In this context, $\tD $ and $\tD ^e$ are, respectively, the
interior and exterior Dirichlet traces. Analogously, for $s>\frac{3}{2}$,
we denote the interior and exterior Neumann traces on $\Omega$ as
\begin{align*}
\tN :H^s_{\theta}(\Omega)\to H^{s-\frac{3}{2}}_\theta(\Gamma),\quad \tN^{e} :H^s_{\theta}(\cG\setminus\overline\Omega^{\cG})\to H^{s-\frac{3}{2}}_\theta(\Gamma),
\end{align*}
extending the normal derivative $u\mapsto \nabla u|_{\Gamma}\cdot\bn$,
where $\bn$ is---for both traces---the unitary normal
exterior to $\Omega$. 
Moreover, introducing the subspace of
elements of $H^1_{\theta}(\Omega)$ with integrable Laplacian,
\begin{align*}
H^s_{\theta,\Delta}(\Omega):=\left\lbrace u\in H^1_\theta(\Omega)\ :\ \Delta u\in L^2_\theta(\Omega) \right\rbrace,
\end{align*}
the Neumann trace may be extended as
\begin{align*}
\tN :H^1_{\theta,\Delta}(\Omega)\to H^{-\half}_\theta(\Gamma),\quad \tN^{e} :H^1_{\theta,\Delta}(\cG\setminus\overline\Omega^{\cG})\to H^{-\half}_\theta(\Gamma),
\end{align*}
through integration by parts
(cf.~\cite[Section 2]{APJ:2019}). All the previous
results concerning trace operators follow analogously
(with obvious modifications) for both local spaces---in the
case that $\Omega$ is unbounded---and if $\Omega$ is the
bounded space between two non-intersecting periodic
curves $\Gamma_1$ and $\Gamma_2$. Finally, we denote
the following vector operators
\begin{align*}
\bm{\gamma} u:=(\tD u,\tN u)^t,\quad\bm{\gamma}^e u:=(\tD^{e} u,\tN^{e} u)^t\quad\text{and}\quad
[\bm{\gamma} u]_{\Gamma}:=\bm{\gamma}^e u-\bm{\gamma} u,
\end{align*}
as the interior, exterior and jump trace vectors on $\Gamma$, respectively.

\section{Helmholtz problem in periodic layered media}
\label{sec:wavepropagationvolume}
%%%%%%%%%%%%%%%%%%%%%%%%%%%%%%%%%%%%
%%%%%%%%%%%%%%%%%%%%%%%%%%%%%%%%%%%%
\subsection{Geometric Setting}
\label{sec:geo}
%%%%%%%%%%%%%%%%%%%%%%%%%%%%%%%%%%%%
We seek to establish a
boundary integral representation for scattered and
transmitted acoustic or electromagnetic fields resulting from plane waves impinging a multi-layered grating. The domain is described by
$M \in\IN$ finite non-intersecting periodic surfaces
$\{\widetilde{\Gamma}_i\}_{i=1}^{M}$---ordered downwards---separating $M+1$ periodic domains $\{\widetilde\Omega_i\}_{i=0}^{M}$
such that for $0<i<M$ it holds $\partial\widetilde\Omega_i= 
\widetilde\Gamma_i\cup\widetilde\Gamma_{i+1}$,
$\partial\widetilde\Omega_0=\widetilde\Gamma_1$ and
$\partial\widetilde\Omega_{M}=\widetilde\Gamma_M$ (see Figure \ref{fig:example}). Moreover,
while all domains $\{\widetilde\Omega_i\}_{i=0}^M$ are unbounded
along $\bm{e}_1$---due to their periodicity---only two of
them, namely $\widetilde\Omega_0$ and $\widetilde\Omega_M$,
are unbounded in the second spatial dimension (along $\bm{e}_2$). The restrictions
of the aforementioned domains and surfaces to the periodic cell
$\cG$ are denoted by:
$$\Omega_i:=\widetilde\Omega_i\cap\cG
\quad \forall\ i\in\{0,\ldots, M\}, \quad\Gamma_j:=\widetilde\Gamma_j\cap\cG
\quad \forall\ j\in\{1,\ldots, M\}.$$ Additionally, we fix $H>0$ so that
\begin{align*}
\bigcup\limits_{i=1}^{M-1}\overline{\Omega}^\cG_{i}\subset \left\lbrace \bx\in\cG\ :\ \modulo{x_2}<H \right\rbrace
\end{align*}
holds. We will assume that the interfaces $\Gamma_i$ , $i \in \{1,\hdots,M\}$ are all Jordan curves of class $\mathcal{C}^\infty$. Furthermore, for each $i\in\{1,\ldots,M\}$,
the exterior and interior trace operators on $\Gamma_i$ are understood as
\begin{gather*}
\tD ^e:H^1_\theta(\Omega_{i-1})\to H^{\half}_\theta(\Gamma_i),\quad\tD :H^1_{\theta}({\Omega_i})\to H^{\half}_\theta(\Gamma_i),\\
\tN ^e:H^1_{\theta,\Delta}(\Omega_{i-1})\to H^{-\half}_\theta(\Gamma_i)\quad\text{and}\quad\tN :H^1_{\theta,\Delta}({\Omega_i})\to H^{-\half}_\theta(\Gamma_i),
\end{gather*}
and the normal vector on $\Gamma_i$ is chosen to point towards $\Omega_{i-1}$.

\begin{figure}[t]
\centering
\includegraphics[scale=0.4]{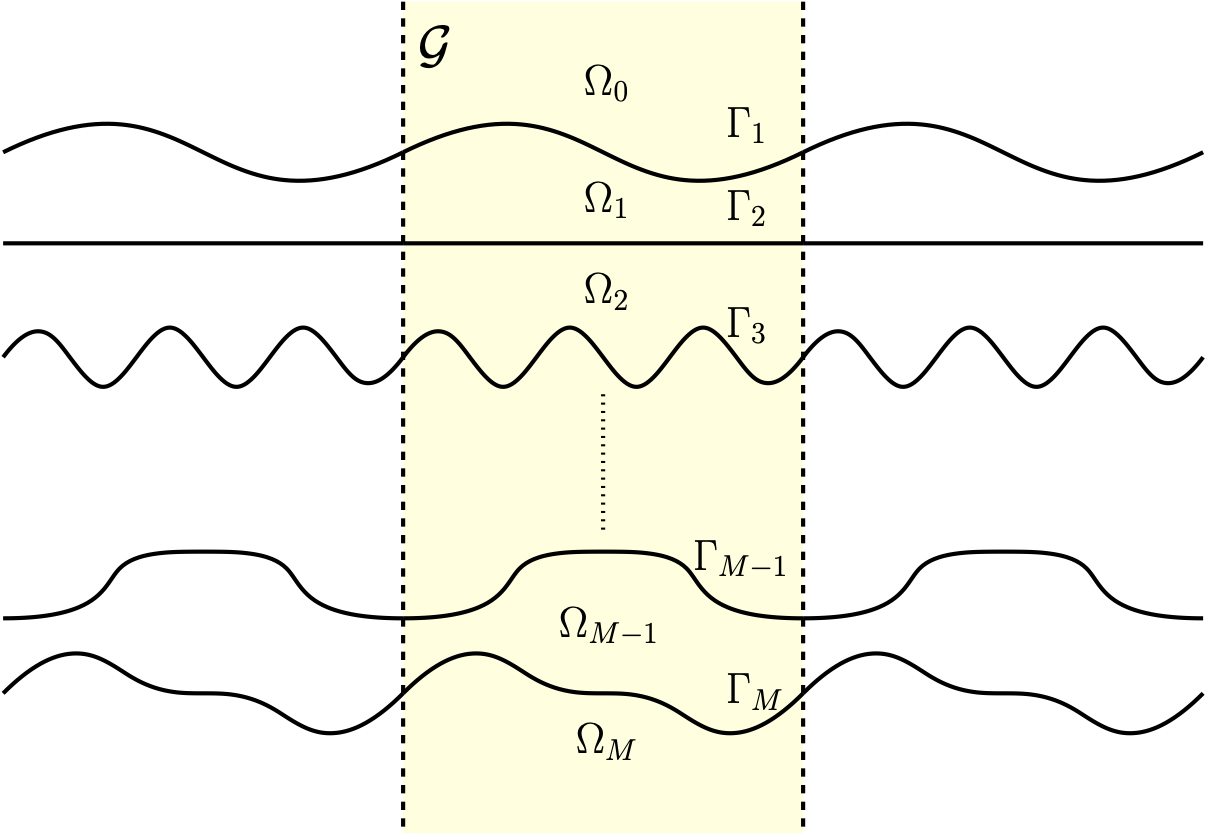}
\caption{Example of a multi-layered grating.
$\cG$ is highlighted and the dotted lines represent its boundaries at $0$ and $2\pi$.}
\label{fig:example}
\end{figure}
%%%%%%%%%%%%%%%%%%%%%%%%%%%%%%%%%%%%
\subsection{Helmholtz transmission problem on periodic media}
\label{ssec:Helmholtz_per}
%%%%%%%%%%%%%%%%%%%%%%%%%%%%%%%%%%%%
For a time-dependence $e^{-\imath\omega t}$ for some frequency $\omega>0$, let the grating described in the previous
subsection be illuminated by an incident plane wave,
\begin{align*}
u^\text{(inc)}(\bx):=e^{\imath \mathbf{k_0}\cdot\bx}=e^{\imath(k_{0,1}x_1+k_{0,2}x_2)},
\end{align*}
where $\mathbf{k}_0=(k_{0,1},k_{0,2})\in\IR^2$. Furthermore, we  denote $k_0:=\modulo{\mathbf{k_0}}$.

For $i=0,\hdots,M$, the material filling each domain $\Omega_i$ is assumed to be homogeneous
and isotropic with refraction index $\eta_i$ --we assume $\eta_0\equiv 1$-- and wavenumber $k_i:=\omega c_i^{-1}=\eta_ik_0$, where $c_i$ is the wave speed in
$\Omega_i$. Throughout this section, we fix $\theta$
as the unique real in $[0,1)$ such that $\theta=k_{0,1}+n$ for some integer $n$ and, for all $j\in\IZ$, we define 
\begin{align}
\label{eq:betacofs}
\beta^{(0)}_j:=\begin{cases} \sqrt{k_0^2-j_{\theta}^2}&\text{ if }k_0^2-j_{\theta}^2\geq0\\
\imath\sqrt{j_{\theta}^2-k_0^2}&\text{ if }k_0^2-j_{\theta}^2<0
\end{cases},\quad
\beta^{(M)}_j:=\begin{cases} \sqrt{k_M^2-j_{\theta}^2}&\text{ if }k_M^2-j_{\theta}^2\geq0\\
\imath\sqrt{j_{\theta}^2-k_M^2}&\text{ if }k_M^2-j_{\theta}^2<0
\end{cases},
\end{align}
where as before $j_\theta = j+\theta$. With these definitions, we can state our volume problem as follows.
\begin{prblm}[Helmholtz transmission problem]
\label{prob:VolProb}
We seek $u^\text{(tot)}$ defined as
\begin{align*}
u^{\text{(tot)}}:=\begin{cases}
u^{\text{(inc)}}+u_0& \text{in }\Omega_0,\\
u_i &\text{in }\Omega_\text{ for }i\in\{1,\hdots, M\},
\end{cases}
\end{align*}
where $u_0\in H^1_{\theta,\text{loc}}(\Omega_0)$, $u_{M}\in H^1_{\theta,\text{loc}}(\Omega_{M})$ and $u_{i}\in H^1_{\theta}(\Omega_{i})$ for all $1\leq i \leq M-1$,
such that 
\begin{subequations}
\label{eq:HelProb}
\begin{align}
\label{eq:Heleq}
&-(\Delta+k_i^2) u^{\text{(tot)}} = 0\quad\text{in }\Omega_i \cap \left\lbrace\bx\in\cG\; :\; |x_2| \leq H\right\rbrace,\quad \forall\ i\in\{0,\hdots,M\},
\end{align}
\begin{align}
\label{eq:transm}
&\left[\bm{\gamma} u^{\text{(tot)}}\right]_{\Gamma_i} =0\quad\text{on }\Gamma_i,\quad \forall\ i\in\{1,\hdots,M\},
\end{align}
\begin{align}
\label{eq:radup}
&u_0(\bx) =\sum\limits_{j\in\IZ}u_{j}^{(0)}e^{\imath\left(\beta^{(0)}_j(x_2-H)+j_\theta x_1\right)}\quad\text{for } x_2\geq H,
\end{align}
\begin{align}
\label{eq:raddown}
&u_m(\bx) =\sum\limits_{j\in\IZ}u_{j}^{(M)}e^{\imath\left(\beta^{(M)}_j(x_2+H)+j_\theta x_1\right)}\quad\text{for } x_2\leq -H.
\end{align}
\end{subequations}
\end{prblm}

Equation \eqref{eq:transm} represents the continuity of Dirichlet and Neumann traces across each interface. This condition can be generalized to include different transmission coefficients without much effort. The last two conditions, namely \eqref{eq:radup} and \eqref{eq:raddown}, correspond to
radiation conditions for $u_0$ and $u_m$, also known as the Rayleigh-Bloch expansions (cf.~\cite{Nedelec:1991} for a detailed discussion),
where $\{u_{j}^{(0)}\}_{j\in\IZ}$ and
$\{u_{j}^{(M)}\}_{j\in\IZ}$ are the corresponding Rayleigh coefficients.

Through an analogous analysis to
that presented in \cite[Section 3]{Elschner:1998}, one finds
that---for a fixed choice of geometries $\{\Gamma_i\}_{i=1}^{M}$
and refraction indices $\{\eta_i\}_{i=1}^{M}$---Problem
\ref{prob:VolProb} has a unique solution for all but a
countable number of wavenumbers $k_0$ as
all wavenumbers $k_i$ for $i\in\{1,\hdots,M\}$ depend
on $k_0$.

\begin{assumption}\label{ass:k0}
The wavenumber $k_0$ is such that Problem \ref{prob:VolProb}
has a unique solution.
\end{assumption}

We shall make no further analysis of the volume problem as stated above, 
and limit ourselves to \cite{Bao:1997,Bao:1995aa,kirsch1993diffraction,Nedelec:1991,
Starling:1994,zhang1998uniqueness,Elschner:1998} and references therein
for more detailed analyses of the radiation condition of similar problems.

%%%%%%%%%%%%%%%%%%%%%%%%%%%%%%%%%%%%
\section{Boundary integral equations}
\label{sec:BIEBIO}
%%%%%%%%%%%%%%%%%%%%%%%%%%%%%%%%%%%%
Following our previous work \cite{APJ:2019}, we 
introduce the quasi-periodic Green's function and recall
some relevant properties. We then define the quasi-periodic single and double layer potentials
and BIOs spanning from taking their respective traces on
the periodic boundaries $\{\Gamma_i\}_{i=1}^{M}$. To 
conclude this section, we present an integral representation
for the fields $\{u_i\}_{i=0}^{M}$ and a proof of unisolvency
for the corresponding BIE.
As before, $\theta$ will denote the quasi-periodic shift, which is assumed to be in $[0,1)$.

\subsection{Quasi-Periodic Fundamental Solution}
Consider a positive wavenumber $k\in\IR$, we recall the definition of the Rayleigh-Wood frequencies.
\begin{dfntn}
We say $k>0$ is a Rayleigh-Wood frequency, if there is $j\in\IZ$, such that 
\begin{align}
\label{eq:woodgeneraleq}
|j+\theta|=k,
\end{align}
where $\theta$ is the previously fixed quasi-periodic shift. 
\end{dfntn}

These frequencies correspond to values where the quasi-periodic Green's function can not be represented in a traditional manner. While a number of alternatives have been developed to circumvent this issue (e.g., \cite{bruno2017rapidly,bruno2017three,Barnett:2015}) their analysis is out of the scope of our current work. Hence, in what follows, we will work under the following assumption over the wavenumber $k$. 

\begin{assumption}\label{ass:woodgeneral}
The wavenumber $k>0$ is not a 
Rayleigh-Wood frequency for the given $\theta \in [0,1)$ .
\end{assumption} 
Under Assumption \ref{ass:woodgeneral} we can define the $\theta$-quasi-periodic Green's function as (cf.~\cite{Nedelec:1991,Linton:1998} and
references therein)
\begin{align}
G^k_{\theta}(\bx,\by):=\lim\limits_{m\rightarrow\infty}\sum\limits_{n=-m}^m
e^{-\imath 2{\pi}n\theta}G^k(\bx+2\pi n\bm{e}_1,\by), \label{eq:qpgreenfunc}
\end{align}
for all $\bx$, $\by$ in $\IR^2$ such that $\bx-\by\neq 2\pi n\bm{e}_1$ for all $n\in\IZ$,
wherein $G^k(\bx,\by)$ is the fundamental solution for the Helmholtz equation
with wavenumber $k$, namely,
\begin{align*}
G^k(\bx,\by)= \frac{\imath}{4} H_0^{(1)}(k\| \bx-\by\|_{\IR^2}),
\end{align*}
where $H_0^{(1)}(\cdot)$ denotes the zeroth-order first kind
Hankel function. Moreover, the quasi-periodic Green's function is a fundamental solution of the Helmholtz equation in the following sense: 
\begin{align*}
-(\Delta_{{\by}} {+} k^2)G^k_{\theta}(\bx,\by)=\sum\limits_{n\in\IZ}\delta(\bx+2\pi n\bm{e}_1)e^{\imath2\pi n\theta}
\end{align*}
for all $\bx\in\IR^2$ and satisfies the radiation condition
specified in the preceding section
(cf.~\cite[Proposition 3.1]{Nedelec:1991}).

\begin{rmrk}
If Assumption \ref{ass:woodgeneral} is not met, the sum in \eqref{eq:qpgreenfunc} fails to converge for any pair of $\bm{x}$, $\bm{y}\in\IR^2$. 
\end{rmrk}

%%%%%%%%%%%%%%%%%%%%%%%%%%%%%%%%%%%%
\subsection{Layer Potentials and Boundary Integral Operators}
\label{sec:BIOs}
%%%%%%%%%%%%%%%%%%%%%%%%%%%%%%%%%%%%
On this section, we will assume a given boundary $\Gamma$ satisfying the following assumption.
\begin{assumption}
\label{ass:single:operator}
Given $r \in [0,\infty]$, the interface $\Gamma$ is a Jordan curve of class $\mathcal{C}^{r,1}$. 
\end{assumption}
Moreover, we denote by $\Omega$ the part of $\cG$ below $\Gamma$ (see Figure \ref{fig:example}). For $\phi\in\cD_\theta(\Gamma)$
we define the single and double layer potentials as
\begin{align}\label{eq:SLDLintro}
\mathsf{SL}_{\theta,\Gamma}^{k}\phi(\bx):=\int_{\Gamma}G^k_\theta(\bx,\by)\phi(\by)\;\d\!\by,\quad\mathsf{DL}_{\theta,\Gamma}^{k}\phi(\bx):=\int_{\Gamma}\gamma_{n,\by}G^k_\theta(\bx,\by)\phi(\by)\;\d\!\by,
\end{align}
where $\gamma_{n,\by}$ denotes the interior (with respect to $\Omega$)
Neumann trace operator acting on functions with argument $\by$. 
\begin{lmm}[Theorems 4.7 and 4.10 in \cite{APJ:2019}]
\label{lemma:Contlayer}
Let $k$ and $\Gamma$ be as in Assumptions \ref{ass:woodgeneral} and \ref{ass:single:operator} with $r\geq 0$, respectively. Then,
the single and double layer potentials can be extended as continuous operators
acting on Sobolev spaces as follows,
\begin{align*}
\mathsf{SL}_{\theta,\Gamma}^k:H^{s-\half}_{\theta}{(\Gamma)}\to H^{s+1}_{\theta,\text{loc}}(\cG)\quad\text{and}\quad
\mathsf{DL}_{\theta,\Gamma}^k:H^{s+\half}_{\theta}{(\Gamma)}\to H^{s+1}_{\theta,\text{loc}}(\cG\setminus\Gamma), \quad \text{for $s<\half$.}
\end{align*}

\end{lmm}
We then define BIOs by taking traces of the layer potentials as follows
\begin{equation}\label{eq:BIOintro}
\begin{aligned}
\V^{k}_{\theta,\Gamma}&:=\tD \mathsf{SL}_{\theta,\Gamma}&&{\Kk^{\prime k}_{\theta,\Gamma}}:=\tN \mathsf{SL}_{\theta,\Gamma}+\half\mathsf{I},\\
\W^{k}_{\theta,\Gamma}&:=-\tN \mathsf{DL}_{\theta,\Gamma}&&{\Kk^{k}_{\theta,\Gamma}}:=\tD \mathsf{DL}_{\theta,\Gamma}-\half\mathsf{I}.
\end{aligned}
\end{equation}
Moreover, due to the jump properties of the layer potentials
\cite[Lemma 4.11]{APJ:2019}, the following relations hold:
\begin{equation}
\label{eq:BIOintroExterior}
\begin{aligned}
 \V^{k}_{\theta,\Gamma}&=\tD^{e} \mathsf{SL}_{\theta,\Gamma}, && {\Kk^{\prime k}_{\theta,\Gamma}}=\tN^{e} \mathsf{SL}_{\theta,\Gamma}-\half\mathsf{I},\\
\W^{k}_{\theta,\Gamma}&=-\tN^{e} \mathsf{DL}_{\theta,\Gamma},&&{\Kk^{k}_{\theta,\Gamma}}=\tD^{e} \mathsf{DL}_{\theta,\Gamma}+\half\mathsf{I}.
\end{aligned}
\end{equation}
\begin{rmrk}
When considering interior and exterior traces acting on layer potentials, note
that the normal vector on $\Gamma$ is to be \emph{fixed} so that the only difference
between exterior and interior traces is the direction from which we approach
$\Gamma$. Additionally, note that, having fixed the normal vector to $\Gamma$, the choice
of trace taken in the definition of
$\W^{k}_{\theta,\Gamma}$ is arbitrary
and makes no difference.
\end{rmrk}
\begin{lmm}[Theorem 4.10 in \cite{APJ:2019}]
\label{lemma:ContBIO}
Let $k$ and $\Gamma$ be as in Assumptions \ref{ass:woodgeneral} and \ref{ass:single:operator} with $r\geq 0$, respectively, and let $s<\half$. Then, the BIOs satisfy the following continuity conditions
\begin{align*}
\V^{k}_{\theta,\Gamma}:H^{s-\half}_\theta(\Gamma)\rightarrow H^{s+\half}_{\theta}(\Gamma),\quad
\W^{k}_{\theta,\Gamma}:H^{s+\half}_\theta(\Gamma)\rightarrow H^{s-\half}_{\theta}(\Gamma),\\
{\Kk^{\prime k}_{\theta,\Gamma}}:H^{s-\half}_\theta(\Gamma)\rightarrow H^{s-\half}_{\theta}(\Gamma),\quad
{\Kk^{k}_{\theta,\Gamma}}:H^{s+\half}_\theta(\Gamma)\rightarrow H^{s+\half}_{\theta}(\Gamma).
\end{align*}
\end{lmm}
\subsubsection{Compacteness Properties}
\label{ssec:compacteness}
Until this point, we have established continuity properties of the four BIOs defined in \eqref{eq:BIOintro}.
However, the BIEs we consider in the coming section
require the subtraction of two instances of
the same BIO with different wavenumbers. This will require a
number of results from pseudo-differential operator theory \cite{saranen2013periodic} as well as a 
version of the Rellich theorem on quasi-periodic Sobolev spaces on boundaries.
After our analysis, we will see that the difference between any two of the operators in
\eqref{eq:BIOintro}---with different wavenumbers---will result in a compact operator. 
\begin{thrm}[{Rellich Theorem for quasi-periodic Sobolev spaces}]
\label{trm:Rellich}
Let $s_1$, $s_2$ be real numbers such that $s_1 < s_2$ and $\theta \in [0,1)$. Then, $H^{s_2}_\theta(\Gamma)$ is compactly embedded in $H^{s_1}_\theta(\Gamma)$.
\end{thrm}

\begin{proof}
Follows directly from the definition of the quasi-periodic spaces
and the result for standard Sobolev spaces (see \cite[Theorem 8.3]{kress:2014}).
\end{proof}

\begin{rmrk}
No smoothness assumptions are needed for the proof of the previous theorem. Thus, it can be extended to Lipchitz boundaries for any pair of real numbers $s_1$, $s_2<1$, and potentially less regular cases if we restrict $s_1$, $s_2$ to be non-negative.  
\end{rmrk}

\begin{thrm}[Theorem 6.1.1 in \cite{saranen2013periodic}]
\label{thrm:ppo}
Let  $a : \IR \times \IR \rightarrow \mathbb{C}$ be
a bi-periodic function of class $\C^\infty$ and $S$
be a $2 \pi-$periodic distribution in $\IR$. Consider the following formal operator acting
on a periodic smooth function $u\in\C^{\infty}(\IR)$:
\begin{align}\label{eq:A:formal:definition}
Au(s) = \int_0^{2\pi} S(s-t)a(s,t)u(t) dt\quad \forall\;s \in \IR,
\end{align} 
where integration is to be understood as a duality pairing. Furthermore, let us assume the Fourier coefficients of $S$ to behave as 
$$|S_n| \lesssim |n|^p,$$ for some $p \in\IR$.
Then, for any $s\in\IR$, $A$ in \eqref{eq:A:formal:definition} may be continuously extended as an operator
mapping from $H^s[0, 2\pi]$ to $H^{s-p}[0, 2\pi]$, i.e.,
$$A : H^s[0, 2\pi] \rightarrow H^{s-p}[0, 2\pi].$$
\end{thrm}
We also recall a classical result from Fourier analysis (\emph{c.f.}~\cite{taibleson1967fourier}).
\begin{lmm}
\label{lemma:FourierCofs}
Let $m \in \IN$, $f:\IR \rightarrow \mathbb{C}$ be a periodic $\mathcal{C}^m$-class function
  such that its distributional derivative
of order $m+1$ belongs to $L^1((0,2\pi))$. Then,
its Fourier coefficients $\{f_n\}_{n\in\IZ}$ are such that $$ |f_n| \lesssim {|n|}^{-m-1}.$$ 
\end{lmm}
In order to employ Theorem \ref{thrm:ppo}
we will need to express the quasi-periodic BIOs in
a convenient way: with periodic functions as kernels. Let $k$ and $\Gamma$ be as in Assumptions \ref{ass:woodgeneral} and \ref{ass:single:operator}, respectively. We begin by
considering a periodic version of the fundamental solution
in \eqref{eq:qpgreenfunc}
and its derivatives on $\Gamma$ as
\begin{align}
\label{eq:vkernel}
\widehat{G}^k_\theta (s,t) := 
e^{-\imath  \theta (s-t)}G^k_\theta (\bz(s),\bz(t)),
\end{align}
which may be expressed as 
\begin{align}
\label{eq:vsplit}
\widehat{G}^k_\theta (s,t) = 
S(t-s) J_\theta^k(s,t) + R_\theta^k(s,t),
\end{align}
with
\begin{gather}
S(t):= -\frac{1}{2\pi} \log\left| 2 \sin \frac{|t|}{2} \right|,\label{eq:Sdef}\\
J_\theta^k(s,t) := e^{-\imath \theta(s-t)} \sum_{j = -\infty} ^\infty J_0(k \| \bz(s)+2\pi j \bm{e_1}- \bz(t) \|)e^{-\imath 2\pi j \theta }\chi_{\epsilon}(s-t),\nonumber
\end{gather}
where $J_0(\cdot)$ is the zeroth-first kind Bessel function, $\epsilon\in (0,2\pi)$ and $\chi_\epsilon(\cdot)$
is a smooth function satisfying
\begin{gather*}
\chi_\epsilon(s)=0\quad\text{if}\quad\modulo{s}>\epsilon\quad\text{and}\quad \chi_\epsilon(s)=1\quad\text{if}\quad\modulo{s}<\half\epsilon,
\end{gather*}
and
\begin{gather*}
R^k_\theta(s,t) = \widehat{G}^k_\theta (s,t) - S(t-s)J^k_\theta(s,t).
\end{gather*}
Using known expansions of the Hankel functions (see \cite[9.1.12-9.1.13]{abramowitz1965handbook}) one can check that $R^k_\theta$ belongs to $\mathcal{C}^\infty(\IR \times \IR)$.

Before we proceed any further, it is necessary to introduce a second wavenumber. We will denote $\widetilde{k}>0$ a wavenumber (not necessarily different from $k$) that also satisfies Assumption \ref{ass:woodgeneral}. 
%\jp{
%\begin{assumption}
%\label{ass:double:op}
%Given $\theta \in [0,1)$ the \ra{quasi-}periodic shift, $r \in [0,\infty]$, then the wavenumbers $k, \widetilde{k}$, and the interface $\Gamma$ are such that: 
%\begin{enumerate}
%\item 
%$k$ and $\widetilde{k}$ are positive and they are not Rayleigh-Wood frequencies for $\theta$. 
%\item 
%$\Gamma$ is a Jordan curve of class $\mathcal{C}^{r,1}$.
%\end{enumerate}
%\end{assumption}
%}
\begin{prpstn}
\label{prop:vscompact}
Let $k$ and $\widetilde{k}$ satisfy Assumption \ref{ass:woodgeneral}, and let $\Gamma$ satisfy Assumption \ref{ass:single:operator} with $r= \infty$. Consider $\mathsf{V}_\theta^k$ and $\mathsf{V}_\theta^{\widetilde k}$ the weakly
singular BIOs on $\Gamma$ defined in \eqref{eq:BIOintro} and where we have dropped the $\Gamma$ subscript for brevity.
Both operators may be considered as pseudo-differential operators
of order $-1$, whence
$$\mathsf{V}_\theta^k:H^s_\theta(\Gamma)\to H^{s+1}_\theta(\Gamma), \quad \mathsf{V}_\theta^{\widetilde{k}}:H^s_\theta(\Gamma)\to H^{s+1}_\theta(\Gamma).$$
Moreover, the operator $\mathsf{V}_\theta^{k,\widetilde{k}}:= \mathsf{V}_\theta^k - \mathsf{V}_\theta^{\widetilde{k}}$
can be extended to
$$\mathsf{V}_\theta^{k,\widetilde{k}}:H^{s}_\theta(\Gamma)\to H^{s+3}_\theta(\Gamma), $$
as a bounded linear operator for every $s \in \IR$.
\end{prpstn}

\begin{proof}
That $\mathsf{V}_\theta^k$ (and $\mathsf{V}_\theta^{\widetilde k}$) may be extended
as claimed follows directly from Theorem \ref{thrm:ppo},
the kernel representation \eqref{eq:vsplit} and the decay of the Fourier
coefficients of $S(t)$ in \eqref{eq:Sdef}
(cf.~\cite[Example 5.6.1]{saranen2013periodic}).
Take $\mu \in D_\theta(\Gamma)$, we have that 
\begin{align*}
\left(\mathsf{V}_\theta^{k,\widetilde{k}}(\mu)\circ\bz\right)(s) = 
e^{\imath  \theta s}\int_0^{2\pi}
 \left(\widehat{G}^k_\theta(s,t)-\widehat{G}^{\widetilde{k}}_\theta(s,t)\right)
 e^{-\imath \theta t}(\mu\circ\bz)(t) \| \bz'(t)\| dt
\end{align*}
as a Lebesgue integral. Moreover,
\begin{align}\label{eq:greendif}
\begin{aligned}
\widehat{G}^k_\theta(s,t)-\widehat{G}^{\widetilde{k}}_\theta(s,t) = 
S(t-s)\left(J_\theta^k(s,t) -J_\theta^{\widetilde{k}}(s,t)\right)+
\left(R_\theta^k(s,t) -R_\theta^{\widetilde{k}}(s,t)\right).
\end{aligned}
\end{align}
Employing Lemma \ref{lemma:FourierCofs}, Theorem
\ref{thrm:ppo} and \cite[Equation 9.1.13]{abramowitz1965handbook}
we see that the second term of the right-hand side of \eqref{eq:greendif} gives rise to a bounded
operator from $H^s[0,2\pi]$ to $H^{s+p}[0,2\pi]$
for any $p >0$. On the other hand, the first term in the right-hand
side of \eqref{eq:greendif} may be decomposed as
\begin{align*}
S(t-s)\left(J_\theta^k(s,t) - J_\theta^{\widetilde{k}}(s,t)\right) =
\left(\modulo{\sin(t-s)}^2{S}(t-s)\right)
\left(\frac{J_\theta^k(s,t) - J_\theta^{\widetilde{k}}(s,t)}
{\modulo{\sin(t-s)}^2}\right).
\end{align*}
One can see (cf.~\cite[Equation 9.1.12]{abramowitz1965handbook})
that the term $({J_\theta^k(s,t) -J_\theta^{\widetilde{k}}(s,t)})
\modulo{\sin(t-s)}^{-2}$ belongs to $\mathcal{C}^\infty(\IR\times\IR)$, whereas the term $\modulo{\sin(t-s)}^2{S}(t-s)$
give rise to an operator of order $-3$. In fact, its Fourier transform is
\begin{align*}
-\frac{1}{2\pi}\int_{0}^{2\pi} \sin(t)^2 \log\left| 2 \sin {\frac{t}{2}} \right| e^{\imath n t}dt &=-\frac{1}{2\pi}\int_{0}^{2\pi} \log\left| 2 \sin {\frac{t}{2}} \right| (e^{\imath (n+2)t}+e^{\imath (n-2)t}-2e^{\imath (n)t})dt \\&= O(n^{-3}), 
\end{align*}
 where the last equality follows from \cite[Example 5.6.1]{saranen2013periodic}. Finally, define 
\begin{align*}
\widehat{\mathsf{V}}^{k,\widetilde k}_\theta(\mu)(s):=
e^{-\imath\theta s}{\mathsf{V}}^{k,\widetilde k}_\theta(\mu)\circ \bz(s).
\end{align*} 
Then,
\begin{align}\label{eq:Vktildek}
\|\mathsf{V}_\theta^{k,\widetilde{k}}(\mu)\|_{H^s_\theta(\Gamma)} \cong
\|\mathsf{V}_\theta^{k,\widetilde{k}}(\mu)\circ\bz\|_{H^s_\theta[0,2\pi]} = 
\|\widehat{\mathsf{V}}_\theta^{k,\widetilde{k}}(\mu)(s)\|_{H^s[0,2\pi]}.
\end{align}
We may now bound the last term in \eqref{eq:Vktildek} by Theorem \ref{thrm:ppo}:
$$ \|\widehat{\mathsf{V}}_\theta^{k,\widetilde{k}}(\mu)\|_{H^{s+3}[0,2\pi]}  \lesssim \| \mu \|_{H^{s}_\theta (\Gamma)}. $$
The proof is completed by the density of $\mathcal{D}_\theta(\Gamma)$
in the corresponding Sobolev space. 
\end{proof}
For the hyper-singular BIO, a similar result requires a technical lemma. To this end, let us define the tangential curl operator: 
\begin{align*}
\scurl_\Gamma \varphi := \frac{1}{\norm{\dot{\bm{z}}(t)}{\IR^2}}
\frac{\;\d}{\;\d\! t}(\varphi\circ\bm{z})(t).
\end{align*}
for any $\varphi \in \cD_\theta(\Gamma)$ and where $\bm{z}$ is a suitable (arbitrary) parametrization of $\Gamma$.

\begin{lmm}\label{lemma:intpartsW}
Let $k$ and $\Gamma$ satisfy Assumptions \ref{ass:woodgeneral} and \ref{ass:single:operator} for $r=0$, respectively, and let $\lambda$ and $\varphi$ belong to $\cD_\theta(\Gamma)$. Then, 
\begin{align*}
\p{\mathsf{W}_\theta^k (\lambda) , {\varphi}}_\Gamma = \p{\mathsf{V}_\theta^{k} (\scurl_\Gamma \lambda), \scurl_\Gamma {\varphi}}_\Gamma +
\p{\widecheck{\mathsf{V}}_\theta^k(\lambda), {\varphi}}_\Gamma,
\end{align*}
where $\p{\cdot,\cdot}_\Gamma$ represents the duality product between
$H^s_\theta(\Gamma)$ and $H^{-s}_\theta(\Gamma)$ for any {$s>0$} and
$\widecheck{\mathsf{V}}_\theta^k$ is the extension
by density of the operator {given by}
\begin{align*}
\p{\widecheck{\mathsf{V}}_\theta^k (\lambda), {\varphi}}_\Gamma := -k^2\int_\Gamma \int_\Gamma  \bm{n}(\bm{x}) \cdot \bm{n}(\bm{y}) G_\theta^k(\bm{x},\bm{y})\lambda(\bm{y}) \overline{{\varphi}}(\bm{x})\;\d\!\by\;\d\!\bx.
\end{align*}
\end{lmm}

\begin{proof}
Notice that for $\lambda$, $\varphi$  in $\cD_\theta(\Gamma)$, it holds that
\begin{align*}
\p{\scurl_\Gamma \lambda, {\varphi}}_\Gamma = \int_{0}^{2\pi} 
\frac{\;\d\!{(\lambda\circ \bm{z})(t)}}{\;\d\!t}\overline{({\varphi}\circ \bm{z})}(t)\;\d\!t =
-\int_{0}^{2\pi} \frac{\;\d\!{\overline{({\varphi}\circ \bm{z})}(t)}}{\;\d\!t}
(\lambda\circ \bm{z})(t) \;\d\!t,
\end{align*}
{where the border terms cancel each other out due to the quasi-periodicity of $\lambda$ and $\varphi$.}
Hence, the result for quasi-periodic functions follows
\emph{verbatim} from the standard case (see,
for instance, \cite[Theorem 6.15]{Steinbach:2007aa}).
\end{proof}

\begin{crllr}
Under the {assumptions} of Proposition \ref{prop:vscompact}, consider $\mathsf{W}_\theta^{k}$, and  $\mathsf{W}_\theta^{\widetilde{k}}$, the hyper-singular operators defined as in \eqref{eq:BIOintro} and where we drop the $\Gamma$ subscript. The  operator $\mathsf{W}_\theta^{k,\widetilde{k}} := \mathsf{W}_\theta^k- \mathsf{W}_\theta^{\widetilde{k}}$ can be
extended to
\begin{align*}
\mathsf{W}_\theta^{k,\widetilde{k}} : H^s_\theta(\Gamma) \to H^{s+1}_\theta(\Gamma),
\end{align*}
as a bounded linear operator for every $s \in \IR$.
\end{crllr}

\begin{proof}
Let $\lambda$,  ${\varphi}$ in $\mathcal{D}_\theta(\Gamma)$. By Lemma \ref{lemma:intpartsW}, we have that
\begin{align*}
\p{\mathsf{W}_\theta^{k,\widetilde{k}} (\lambda),{\varphi}}_\Gamma = 
\p{\mathsf{V}_\theta^{k,\widetilde{k}} (\scurl_\Gamma \lambda), \scurl_\Gamma {\varphi}}_\Gamma +
\p{(\widecheck{\mathsf{V}}_\theta^k-\widecheck{\mathsf{V}}_\theta^{\widetilde{k}})(\lambda), {\varphi}}_\Gamma.
\end{align*}
Using Proposition \ref{prop:vscompact}, one obtains
\begin{align*}
\modulo{\p{\mathsf{W}_\theta^{k,\widetilde{k}} (\lambda),{\varphi}}_\Gamma} \lesssim 
\|\scurl_\Gamma \lambda \|_{H^{s-1}_\theta(\Gamma)} \| \scurl_\Gamma {\varphi}\|_{H^{-s-2}_\theta(\Gamma)} + \|\lambda\|_{H^{s}_\theta(\Gamma)}
\|{\varphi}\|_{H^{-s-1}_\theta(\Gamma)}.
\end{align*}
Where the inequality for the second term of the right-hand side is obtained using that both $(\widecheck{\mathsf{V}}_\theta^k,\widecheck{\mathsf{V}}_\theta^{\widetilde{k}})$ are operators of order $-1$ (this follow from Theorem  \ref{thrm:ppo} and \cite[Example 5.6.1]{saranen2013periodic}. 
Then, since the $\scurl_\Gamma$ operator is a first-order differential operator, it holds that
\begin{align*}
|\p{\mathsf{W}_\theta^{k,\widetilde{k}} (\lambda),{\varphi}}_\Gamma| \lesssim 
 \|\lambda\|_{H^{s}_\theta(\Gamma)}
\|{\varphi}\|_{H^{-s-1}_\theta(\Gamma)},
\end{align*}
and the result follows by a duality argument and recalling the
density of $\cD_\theta(\Gamma)$ in our quasi-periodic Sobolev spaces.
\end{proof}

We now consider the Dirichlet traces of the double layer potential and its adjoint, defined in Section \ref{sec:BIOs} as the principal value integrals,
\begin{align*}
(\mathsf{K}'^k_\theta(\mu)\circ \bm{r})(s)  = \dashint_{0}^{2\pi} 
{\mathcal{K}'}^k_\theta (s,t) {(\mu\circ\bm{z})}(t)\|\dot{\bm{z}}(t)\|\;\d\!t ,\\
(\mathsf{K}_\theta^k(\lambda)\circ\bm{r})(s)  = \dashint_{0}^{2\pi} 
{\mathcal{K}}^k_\theta (s,t) {(\lambda\circ\bm{z})}(t)\|\dot{\bm{z}}(t)\|\;\d\!t,
\end{align*}
for which we have dropped the $\Gamma$ index momentarily, and where the kernels are given by \cite[Section 3]{bruno2014rapidly}:
\begin{align*}
\begin{split}
{\mathcal{K}'}^k_\theta (s,t) := -\frac{\imath k}{4} \sum_{j=-\infty}^\infty\Bigg{(} \frac{H_1^{(1)}(k\| 
\bm{z}(s) +2\pi j \bm{e_1} - \bm{z}(t)\|)}
{\| 
\bm{z}(s) +2\pi j \bm{e_1} - \bm{z}(t)\|}
 e^{-\imath 2 \pi j \theta}\ \times \\
(\bm{z}(s) +2\pi j \bm{e_1} - \bm{z}(t)) \cdot 
\bm{n}(\bm{z}(s))\Bigg{)} ,
\end{split} \\
\begin{split}
{\mathcal{K}}^k_\theta (s,t) := \frac{\imath k}{4} \sum_{j=-\infty}^\infty\Bigg{(} \frac{H_1^{(1)}(k\| 
\bm{z}(s) +2\pi j \bm{e_1} - \bm{z}(t)\|)}
{\| 
\bm{z}(s) +2\pi j \bm{e_1} - \bm{z}(t)\|}
 e^{-\imath 2 \pi j \theta}\ \times  \\ 
 (\bm{z}(s) +2\pi j \bm{e_1} - \bm{z}(t)) \cdot \bm{n}(\bm{z}(t))\Bigg{)}.
\end{split}
\end{align*}
where $\bm{n}$ denotes the unitary normal vector {exterior to $\Omega$ (recall $\Gamma:=\partial^{\cG}\Omega$)}. These can be written as \cite[Equation 9.1.11]{abramowitz1965handbook}
\begin{align}
\label{dl:splits}
\begin{gathered}
{\mathcal{K}'}^k_\theta (s,t) = S_1(t-s) J^{k}_{1,\theta}(s,t) +R_{1,\theta}^{k}(s,t)\\
{\mathcal{K}}^k_\theta (s,t) = S_1(t-s) J_{2,\theta}^k(s,t) +R_{2,\theta}^k(s,t),
\end{gathered}
\end{align}
wherein
\begin{align*}
S_1(t-s) &:=  -\frac{1}{2\pi} \log\left( 2 \sin{\left(\half|t-s|\right)} \right) |\sin(t-s)|^2,\\
J_{1,\theta}^k(s,t) &:=  -k \sum_{j=-\infty}^\infty\Bigg{(}  \frac{J_1(k \| \bm{z}(s) +  2\pi j \bm{e_1}-\bm{z}(t)\|)}
{\| \bm{z}(s) +  2\pi j \bm{e_1}-\bm{z}(t)\|} e^{-\imath 2\pi  j \theta}\times\\
 &\qquad\qquad\frac{(\bm{z}(s) +2\pi j \bm{e_1} - \bm{z}(t)) \cdot \bm{n}(\bm{z}(s))}{|\sin(t-s)|^2} \chi_\epsilon(t-2\pi j -s)\Bigg),
\\
J_{2,\theta}^k(s,t) &:=  k\sum_{j=-\infty}^\infty \Bigg{(} \frac{J_1(k \| \bm{z}(s) +  2\pi j \bm{e_1}-\bm{z}(t)\|)}
{\| \bm{z}(s) +  2\pi j \bm{e_1}-\bm{z}(t)\|} e^{-\imath 2\pi  j \theta}\times\\
&\qquad\qquad \frac{(\bm{z}(s) +2\pi j \bm{e_1} - \bm{z}(t)) \cdot \bm{n}(\bm{z}(t))}{|\sin(t-s)|^2} \chi_\epsilon(t-2\pi j -s)\Bigg{)},
\end{align*}
and
\begin{align*}
R_{1,\theta}^{k}(s,t) &:= {\mathcal{K}'}^k_\theta (s,t) - S_1(t-s) J^{k}_{1,\theta}(s,t), \\
R_{2,\theta}^k(s,t) &:= {\mathcal{K}}^k_\theta (s,t) - S_1(t-s) J_{2,\theta}^k(s,t).
\end{align*}
As in the proof of Proposition \ref{prop:vscompact}, we have that $ |S_{1,n}| \lesssim n^{-3},$ whence, arguing as in Proposition \ref{prop:vscompact}, we have the following result.

\begin{prpstn}
For $k$ and $\Gamma$ as in Assumptions \ref{ass:woodgeneral} and \ref{ass:single:operator} with $r= \infty$, respectively, and for any $s \in \IR$, it holds that
\begin{gather*}
\mathsf{K}'^k_\theta : H^s_\theta(\Gamma) \rightarrow H^{s +3}_\theta(\Gamma),\quad
\mathsf{K}_\theta^k : H^{s}_\theta(\Gamma) \rightarrow H^{s +3}_\theta(\Gamma),
\end{gather*}
are bounded and linear operators.
\end{prpstn}

As in the case of the weakly and hyper-singular operator, we define:
\begin{gather*}
{\mathsf{K}'}_\theta^{k,\widetilde{k}} := {\mathsf{K}'}^k_\theta - {\mathsf{K}'}_\theta^{\widetilde{k}},\quad
\mathsf{K}_\theta^{k,\widetilde{k}} := \mathsf{K}_\theta^k - \mathsf{K}_\theta^{\widetilde{k}}.
\end{gather*}
Finally, we obtain our compactness result.
\begin{prpstn}
\label{prop:biecompacteness}
Let $k$ and $\widetilde{k}$ satisfy Assumption \ref{ass:woodgeneral}, let $\Gamma$ be as in Assumption \ref{ass:single:operator} with $r = \infty$. Then, for $s \in \IR$, the following operators 
\begin{gather*}
\mathsf{V}_\theta^{k,\widetilde{k}} : H_\theta^{s}(\Gamma) \rightarrow  H_\theta^{s+3-\epsilon}(\Gamma), \quad
\mathsf{W}_\theta^{k,\widetilde{k}} : H_\theta^{s}(\Gamma) \rightarrow  H_\theta^{s+1-\epsilon}(\Gamma), \\
\mathsf{K}_\theta^{k,\widetilde{k}} : H_\theta^{s}(\Gamma) \rightarrow  H_\theta^{s+3-\epsilon}(\Gamma), \quad
{\mathsf{K}'}_\theta^{k,\widetilde{k}} : H_\theta^{s}(\Gamma) \rightarrow  H_\theta^{s+3-\epsilon}(\Gamma),
\end{gather*}
are compact for every $\epsilon>0$.
\end{prpstn}

\begin{proof}
The result is direct from the mapping properties shown and Theorem \ref{trm:Rellich}.
\end{proof}

Lastly, we require the compactness of the operator resulting
form taking traces of the single and double layer operators acting on densities lying on a boundary $\Gamma_1$ over another
$x_1$-periodic curve, say $\Gamma_2$, that does not intersect with $\Gamma_1$.
Let us denote by $\gamma^2_d, \gamma^2_n$ Dirichlet and Neumann traces
over $\Gamma_2$, respectively. Then, by an application of Lemma
\ref{lemma:FourierCofs}, Theorem \ref{trm:Rellich} and Theorem
\ref{thrm:ppo}, we obtain the following result. 

\begin{prpstn}
\label{prop:crosscompacteness}
Let $k$ satisfy Assumption \ref{ass:woodgeneral}. If $\Gamma_1$ and $\Gamma_2$ are $x_1$-periodic $\mathcal{C}^\infty$-{Jordan} curves then the application of the following traces to the layer potentials:
\begin{align*}
\tD^2 \mathsf{SL}_{\theta,\Gamma_1}^{k} : H^{s_1}_\theta(\Gamma_1) \rightarrow H_\theta^{s_2}(\Gamma_2),\quad
\tN^2 \mathsf{SL}_{\theta,\Gamma_1}^{k} : H^{s_1}_\theta(\Gamma_1) \rightarrow H_\theta^{s_2}(\Gamma_2),\\
\tD^2 \mathsf{DL}_{\theta,\Gamma_1}^{k} : H^{s_1}_\theta(\Gamma_1) \rightarrow H_\theta^{s_2}(\Gamma_2),\quad
\tN^2 \mathsf{DL}_{\theta,\Gamma_1}^{k} : H^{s_1}_\theta(\Gamma_1) \rightarrow H_\theta^{s_2}(\Gamma_2),
\end{align*}
are compact operators for any choice of $s_1$, $s_2\in\IR$.
The result holds regardless of the direction from which the traces are taken.
\end{prpstn}

\begin{rmrk}
\label{rem:limitedreg}
For the main results in this section, we have assumed the
interfaces to be of class $\mathcal{C}^\infty$. While this simplifies the analysis, we could obtain similar results
with less {stringent} regularity {requirements}. Consider $k$  and $\widetilde{k}$ satisfying Assumption \ref{ass:woodgeneral} and $\Gamma$ as in Assumption \ref{ass:single:operator} with $r \in [1,\infty)$, and the weakly-singular operator $V^k_\theta$ (where we have omitted the $\Gamma$ sub-index {momentarily}). {The expression in \eqref{eq:vsplit} still holds for the kernel of $V^k_\theta$, but $R^k_\theta$ and $J^k_\theta$ would be only of class $\mathcal{C}^{r,1}$, instead of arbitrarily} smooth. {Corollary 6.1.1 and Lemma 6.1.3 in} \cite{saranen2013periodic} imply the same results of Propositions \ref{prop:vscompact} and \ref{prop:biecompacteness} for $s$ in a range limited by $r$. 
%\footnote{For the weakly-singular operator following these results the continuity holds for $|s| < r - \frac{1}{2}$, which is more restricted that the result showed in \cite{APJ:2019}[Theorem 4.10], and hence the presented procedure is not the best best suited for low regularity cases.}
\end{rmrk}

\begin{rmrk}
As previously mentioned, we have limited ourselves to extend the classical mapping results of the boundary integral operators to the context of quasi-periodic spaces. For the classical result see, for example, \cite[Theorem 2.1]{nystromconvergence}. 
\end{rmrk}

\subsection{Boundary Integral Formulation}
We recall the notation and geometry configuration introduced in Section \ref{sec:wavepropagationvolume}, that is: 
\begin{enumerate} 
\item 
$u^{(\text{inc})}$ denotes a plane incident wave with wavenumber $k_0$, which is assumed to be quasi-periodic with shift $\theta \in [0,1)$. 
\item 
$\{\Gamma_i\}_{i=1}^M$ denotes a set of $M\in\IN$ non-intersecting $\mathcal{C}^{r,1}$-Jordan curves, with $r \in [1,\infty]$, ordered {downwards}.
\item 
$\{ \Omega_i\}_{i=0}^M$ denotes a set of $M+1$ open domains, ordered downwards with boundaries 
\begin{align*}
\partial^\cG \Omega_0 = \Gamma_1,\qquad
\partial^\cG \Omega_i  = \Gamma_{i} \cup \Gamma_{i+1} \quad \forall\; i\in\{1,\hdots,M-1\}, \qquad
\partial^\cG \Omega_M = \Gamma_M.
\end{align*}
\item $\{\eta\}_{i=1}^M$ denotes a parameter set such that the wavenumber in $\Omega_i$ is given by $k_i = \eta_i k_0$ for $i\in\{1,\hdots,M\}$. 
\end{enumerate}

\begin{assumption}
\label{ass:bieassump} For the given shift, $\theta$, the wavenumber $k_0$ and the parameters $\{\eta_i\}_{i=1}^{M}$
are such that neither $k_0$ nor the wavenumbers $k_i=\eta_ik_0$
are Rayleigh-Wood frequencies.
\end{assumption}

Following the notation of Problem \ref{prob:VolProb}, the scattered field---defined as the total field $u^{(\text{tot})}$ minus
the incident field $u^{(\text{inc})}$---is written as
\begin{align*}
u^{\text{(sc)}}:=%\begin{cases}
u_i\quad\text{in }\Omega_i,\;\text{for }i\in\{0,\hdots, M\}.
%\end{cases}.
\end{align*}
Under Assumption \ref{ass:bieassump}, we make the following representation \emph{Ansatz} for the scattered field:
\begin{align*}
u^{\text{(sc)}}=\begin{cases}
\mathsf{SL}^{k_0}_{\theta,\Gamma_1}(\mu_1) - \mathsf{DL}^{k_0}_{\theta,\Gamma_1}(\lambda_1) &\text{in }\Omega_0,\\
\begin{aligned}
&\mathsf{SL}^{k_i}_{\theta,\Gamma_i}(\mu_i) - \mathsf{DL}^{k_i}_{\theta,\Gamma_i}(\lambda_i)+\\ &\mathsf{SL}^{k_i}_{\theta,\Gamma_{i+1}}(\mu_{i+1})-\mathsf{DL}^{k_i}_{\theta,\Gamma_{i+1}}(\lambda_{i+1})
\end{aligned} &\text{in }\Omega_i,\;\text{for }i\in\{1,\hdots,M-1\}\\
\mathsf{SL}^{k_M}_{\theta,\Gamma_m}(\mu_m) - \mathsf{DL}^{k_M}_{\theta,\Gamma_m}(\lambda_m) &\text{in }\Omega_m,
\end{cases},
\end{align*}
where, for each $i\in\{1,\hdots,M\}$, the boundary data $\lambda_i$ and $\mu_i$ are assumed to belong to $H^s_\theta(\Gamma_i)$ for some possibly different values of $s\in\IR$, i.e., $s$ may be different for each boundary datum. $\mathsf{SL}^{k_{j}}_{\theta,\Gamma_i}$ and $\mathsf{DL}^{k_{j}}_{\theta,\Gamma_i}$ are, respectively, the
single and double layer potentials of wavenumber $k_{j}$ on $\Gamma_i$.

As shorthand, in what follows, we denote, for each $i\in\{1,\hdots, M\}$,
\begin{gather*}
\Lambda_i  := (\lambda_i, \mu_i)^t,\quad
\boldsymbol{\bm{\mathsf{L}}}_{\theta,\Gamma_i}^{k} \Lambda_i := \mathsf{SL}^{k}_{\theta,\Gamma_i}(\mu_i) - \mathsf{DL}_{\theta,\Gamma_i}^{k}(\lambda_i),
\end{gather*} 
where $\lambda_i$ and $\mu_i$ are defined over $\Gamma_i$.
For  $s_1,s_2 \in \IR$, we define the Cartesian product spaces:
\begin{align*}
{\cV}_{\theta,\Gamma_i}^{{s_1,s_2}} := H^{{s_1}}_\theta(\Gamma_i) \times
H^{{s_2}}_\theta(\Gamma_i)\quad \text{for } i =0,\hdots,M\qquad\mbox{and}\qquad
\bm{\cV}_\theta^{s_1,s_2} := \prod\limits_{i=1}^M \mathcal{V}_{\theta,\Gamma_i}^{s_1,s_2},
\end{align*}
where all of these spaces are equipped with their natural graph inner products. For each $i\in\{1,\hdots,M\}$ let us define the following operators:
\begin{equation}\label{eq:self:interaction}
\mathsf{A}_i\Lambda_i :=
 \begin{pmatrix}
 -\mathsf{K}^{k_{i-1},k_i}_{\theta,\Gamma_i}(\lambda_i) +  \mathsf{V}^{k_{i-1},k_i}_{\theta,\Gamma_i}(\mu_i) \\
  \mathsf{W}^{k_{i-1},k_i}_{\theta,\Gamma_i}(\lambda_i) +  {\mathsf{K}'}^{k_{i-1},k_i}_{\theta,\Gamma_i}(\mu_i)
 \end{pmatrix},
\end{equation}
{corresponding to self-interactions between the potentials defined
over each $\Gamma_i$ with themselves. Analogously}, for $i$, $j\in\{1,\hdots,M\}$, we define the following operators:
\begin{equation}\label{eq:cross:interaction}
 \mathsf{B}_{i,j}\Lambda_j := \begin{cases}
 \begin{pmatrix}
 -\tD ^i \mathsf{DL}^{k_{\min\{i,j\}}}_{\theta,\Gamma_{j}}(\lambda_j) + \tD ^i \mathsf{SL}^{k_{\min\{i,j\}}}_{\theta,\Gamma_{j}}(\mu_j) \\
-\tN ^i \mathsf{DL}^{k_{\min\{i,j\}}}_{\theta,\Gamma_{j}}(\lambda_j) + \tN ^i \mathsf{SL}^{k_{\min\{i,j\}}}_{\theta,\Gamma_{j}}(\mu_j) \end{pmatrix} & \text{if} \ \modulo{i-j}=1\\
\qquad\qquad\qquad\qquad\boldsymbol{0} & \text{a.o.c.} %\mathsf{B}_{{j+1},j}(\lambda, \mu)^t :=
% \begin{pmatrix}
% -\tD ^{j+1} \mathsf{DL}^{k_j}_{\theta,\Gamma_j}(\lambda) + \tD ^{j+1} \mathsf{SL}^{k_j}_{\theta,\Gamma_j}(\mu) \\
%-\tN ^{j+1} \mathsf{DL}^{k_j}_{\theta,\Gamma_j}(\lambda) + \tN ^{j+1} \mathsf{SL}^{k_j}_{\theta,\Gamma_j}(\mu) \end{pmatrix}.
 \end{cases}
\end{equation}
corresponding to interactions between potentials defined over $\Gamma_i$ with those defined over $\Gamma_j$. 
%Notice that these are not zero only if $\vert i-j \vert$ are neighbors \todo{definition?}.
\begin{prpstn}
\label{prop:compactops}
Let Assumption \ref{ass:bieassump} hold and let interfaces $\{\Gamma_i\}_{i=1}^{M}$ be of class $\mathcal{C}^\infty$. Then, the self-interaction operators defined in \eqref{eq:self:interaction} 
$$\mathsf{A}_i : \cV_{\theta,\Gamma_i}^{s_1,s_2} \rightarrow \cV_{\theta,\Gamma_i}^{s_1,s_2}$$ are compact operators for any $s_1$, $s_2\in\IR$ with  $s_2<s_1<s_2+2$. Furthermore, the cross-interaction operators \eqref{eq:cross:interaction}
\begin{align*} 
\mathsf{B}_{i,j} :&  \cV_{\theta,\Gamma_j}^{s_1,s_2} \rightarrow \cV_{\theta,\Gamma_i}^{s_1,s_2},  
\end{align*}
are compact for any choice of $s_1$, $s_2\in\IR$.
\end{prpstn}
\begin{proof}
The first result is directly found using Proposition \ref{prop:biecompacteness}, whereas the second one {follows from} Proposition \ref{prop:crosscompacteness}.
\end{proof}

With the above definitions and using the jump properties of the BIOs, it holds that
\begin{equation}
\left[\gamma u^{\text{(sc)}}\right]_{\Gamma_i}=\mathsf{B}_{i,i-1}\Lambda_{i-1}  + (\mathsf{A}_i-\Id_i)\Lambda_i  -\mathsf{B}_{i,i+1} \Lambda_{i+1}, 
\end{equation}
{for each $i\in\{1,\hdots,M\}$}, where $\Id_i$ corresponds to the identity map over $ \cV_{\theta,\Gamma_j}^{s_1,s_2}$, with $s_1$, $s_2 \in \IR$. {We now introduce the following} operator matrix over $\bm{\cV}_\theta^{{s}_1,{s}_2}$,
\begin{equation}
\label{eq:systemjumps}
\bm{\bm{\mathcal{M}}}:=\begin{pmatrix}
\mathsf{A}_1 -\Id_1 & - \mathsf{B}_{1,2}& 0 & 0 & 0 &\hdots & 0 \\
\mathsf{B}_{2,1} & \mathsf{A}_2-\Id_2& -\mathsf{B}_{2,3}& 0 & 0 & \hdots& 0\\
\vdots  & \vdots &\vdots & \vdots & \vdots& \vdots&\vdots \\ 
0 & 0& \hdots & 0  & \mathsf{B}_{M-1,M-2}&\mathsf{A}_{M-1}-\Id_{M-1} & \mathsf{B}_{M-1,M} \\
0 & 0& \hdots &0 & 0 & \mathsf{B}_{M,M-1}&\mathsf{A}_M-\Id_M
\end{pmatrix}.
\end{equation}
Imposing the boundary conditions of Problem \ref{prob:VolProb} to $u^{\text{(sc)}}$ leads {to the following} system of BIEs.
\begin{prblm}\label{prob:BIEsystem}
Let  Assumption \ref{ass:bieassump} hold and let $s \in \IR$. Define ${s}_1:=s+\half$ and ${s}_2:=s-\half$. We seek $\bm{\Lambda}\in\bm{\cV}^{{s}_1,{s}_2}_{\theta}$ such that
\begin{align*}
\bm{\bm{\mathcal{M}}}\bm{\Lambda}= 
\begin{pmatrix}
-{\bm{\gamma}^{e,1}} u^{\text{(inc)}}\\
0\\
\vdots \\
0
\end{pmatrix}  
\end{align*}
where $\bm{\mathcal{M}}$ corresponds to the operator matrix in \eqref{eq:systemjumps} and $\bm{\gamma}^{e,1}$ corresponds to the exterior trace vector on $\Gamma_1$.
\end{prblm}

In order to ensure the well-posedness of Problem \ref{prob:BIEsystem}, we introduce the following set of auxiliary problems.

\begin{prblm}[Auxiliary problems]
\label{prob:auxprob}
We seek $\{v_i\}_{i=1}^M$ such that $v_i\in H^1_{\theta,\text{loc}}(\cG\setminus\Gamma_i)$
\begin{align}
\begin{aligned}
&-(\Delta +k_{i}^2)v_i (\bm{x}) = 0 \quad \text{in } \left(\Omega_{i-1} \cup \bigcup_{j=0}^{i-2} \overline{\Omega_j}^\mathcal{G} \right) {\cap\{\bx\in\cG\; :\;|x_2| < H\}} ,  \\ 
&-(\Delta +k_{i-1}^2)v_i (\bm{x}) = 0\quad \text{in }\Omega_{i} \cup \left( \bigcup_{j=i+1}^{M} \overline{\Omega_j}^\mathcal{G} \right){\cap\{\bx\in\cG\; :\;|x_2| < H\}},\\
&[\bm{\gamma} v_i]_{\Gamma_i}=0\quad\text{on }\Gamma_i\ ,\\
&v_i(\bx)=\sum\limits_{j\in\IZ}v_{j}^{(i)}e^{\imath\left(\beta^{(0)}_j(x_2-H)+j_\theta x_1\right)}\quad\text{{for} all }x_2\geq H,\\
&v_i(\bx)=\sum\limits_{j\in\IZ}v_{j}^{(i)}e^{\imath\left(\beta^{(M)}_j(x_2+H)+j_\theta x_1\right)}\quad\text{{for} all }x_2\leq -H,
\end{aligned}
\end{align}
for each $i\in\{1,\hdots,M\}$, where $H>0$ is as in Section \ref{sec:geo}, and $\{k_i\}_{i=0}^{M}$ are the wavenumbers in each $\{\Omega_i\}_{i=0}^{M}$, as introduced in Section \ref{sec:wavepropagationvolume}.
\end{prblm}
 By the same analysis {as that} presented in \cite[Section 3.4]{Starling:1994}, each {interface $\Gamma_i$}, $i\in\{1,\hdots,M\}$, potentially adds a countable set of
wavenumbers, $k_0$, such that Problem \ref{prob:auxprob} is unsolvable.
{This justifies the} following Assumption {(recall $k_i=\eta_ik_0$ for all $i\in\{1,\hdots,M\}$)}.

\begin{assumption}\label{ass:auxk0}
Given $\{\eta_i\}_{i=1}^{M}$, the wavenumber $k_0$ is such that the auxiliary Problem \ref{prob:auxprob} has only one solution $\{v_i\}_{i=1}^M$ given by $v_i:= 0$ for all $i\in\{1,\hdots,M\}$.
\end{assumption}

Assumption \ref{ass:auxk0} will force us to discard yet more wavenumbers,
but the set of wavenumbers neglected by Assumptions \ref{ass:k0}
and \ref{ass:auxk0} is still countable.

\begin{thrm}
\label{thm:welposed}
Let the parameters $k_0$ and $\{\eta_i\}_{i=1}^M$ satisfy
Assumption \ref{ass:bieassump} and let the interfaces $\{\Gamma_i\}_{i=1}^{M}$ be $\mathcal{C}^\infty$ periodic Jordan arcs. Further assume Assumptions \ref{ass:k0} and \ref{ass:auxk0} to be satisfied. Then,
Problem \ref{prob:BIEsystem} is well posed for any $s \in \IR$. 
\end{thrm}

\begin{proof}
Note that the operator matrix $\bm{\mathcal{M}}$ may be written as
\begin{align*}\scriptstyle{
\bm{\mathcal{M}}=\begin{pmatrix}
\mathsf{A}_1 & - \mathsf{B}_{1,2}& 0 & 0 & 0 &\hdots & 0 \\
\mathsf{B}_{2,1} & \mathsf{A}_2& -\mathsf{B}_{2,3}& 0 & 0 & \hdots& 0\\
\vdots  & \vdots &\vdots & \vdots & \vdots& \vdots&\vdots \\ 
0 & 0& \hdots & 0  & \mathsf{B}_{M-1,M-2}&\mathsf{A}_{M-1}& \mathsf{B}_{M-1,M} \\
0 & 0& \hdots &0 & 0 & \mathsf{B}_{M,M-1}&\mathsf{A}_M
\end{pmatrix}-
\begin{pmatrix}
\Id_1 & 0 & 0 & 0 &\hdots & 0 \\
0 & \Id_2& 0 & 0 & \hdots& 0\\
\vdots  & \vdots & \vdots & \vdots& \vdots&\vdots \\ 
0 & 0& \hdots & 0 &\Id_{M-1}&0 \\
0 & 0& \hdots & 0 &0&\Id_M
\end{pmatrix}.}
\end{align*}
Then, by the Fredhom alternative, we need only show
uniqueness of Problem \ref{prob:BIEsystem}, as the above tridiagonal block is compact by Proposition \ref{prop:compactops}. The
proof is very similar to that for the classical scattering {problem} of a bounded object in free space
(cf.~\cite[Theorem 3.41]{ColtonKress2013}). 

Let $\bm{\Lambda} \in {\bm{\cV}^{{s}_1,{s}_2}_{\theta}}$, {with $s_1=s+\half$ and $s_2=s-\half$}, be such that $\bm{\bm{\mathcal{M}}}\bm{\Lambda} = 0$.
We define 
\begin{alignat*}{2}
\widetilde{u}_0(\bx) &:= \left(\bm{\mathsf{L}}^{k_0}_{\theta,\Gamma_1}(\Lambda_1)\right)(\bx)\quad &&\forall\ \bx\in\cG\setminus\Gamma_1, \\
\widetilde{u}_i(\bm{x}) &:= \left(\bm{\mathsf{L}}^{k_i}_{\theta,\Gamma_i}(\Lambda_i)\right)(\bx) + 
\left(\bm{\mathsf{L}}^{k_i}_{\theta,\Gamma_i}(\Lambda_{i+1})\right)(\bx)
\quad 
&&\forall\ \bx\in\cG\setminus(\Gamma_i\cup\Gamma_{i+1}),\ \forall\ i\in\{1,\hdots,M-1\},
\\
\widetilde{u}_m(\bm{x}) &:= \left(\bm{\mathsf{L}}^{k_M}_{\theta,\Gamma_m}(\Lambda_m)\right)(\bx)\quad &&\forall\ \bx\in\cG\setminus\Gamma_m,
\end{alignat*}
and further {define}
\begin{align*}
\widetilde{u}(\bm{x}) := \widetilde{u}_i(\bm{x}) \quad \forall\ \bx\in \Omega_i,\quad\forall\ i \in \{0,\hdots,M\},
\end{align*}
which is well defined in each $\Omega_i$, but could potentially
have non-zero jumps across each interface $\Gamma_i$.
Moreover, $\widetilde{u}$ solves the Helmholtz equation with wavenumber $k_i$
in each $\Omega_i$ and satisfies {the appropriate radiation} conditions at infinity \cite[Section 4]{APJ:2019}. Hence, $\widetilde{u}$ solves Problem \ref{prob:VolProb}, and
Assumption \ref{ass:k0} implies $\widetilde{u} \equiv 0$. We continue by defining
the following auxiliary functions
\begin{align*}
v_i (\bm{x}) := \begin{cases} 
\widetilde{u}_i(\bm{x})
 \quad &\forall \bm{x} \in \Omega_{i-1} \cup \left( \bigcup\limits_{j=0}^{i-2} \overline{\Omega_j}^\mathcal{G} \right) \\
-\widetilde{u}_{i-1}(\bm{x}) 
\quad &\forall \bm{x} \in \Omega_{i} \cup \left( \bigcup\limits_{j=i+1}^{M} \overline{\Omega_j}^\mathcal{G} \right) 
\end{cases},\quad\forall\ i\in\{1,\hdots,M\}.
\end{align*}
 It is clear from this definition that 
\begin{align*}
(-\Delta -k_{i}^2)v_i (\bm{x}) = 0, \quad \text{in } \Omega_{i-1} \cup \left( \bigcup_{j=0}^{i-2} \overline{\Omega_j}^\mathcal{G} \right), \\ 
(-\Delta -k_{i-1}^2)v_i (\bm{x}) = 0,\quad \text{in }\Omega_{i} \cup \left( \bigcup_{j=i+1}^{M} \overline{\Omega_j}^\mathcal{G} \right). 
\end{align*}
Furthermore, each $v_i$ satisfies the appropriate radiation conditions at infinity.
Using the jump relationships of BIOs  (see \cite[Lemma 4.11]{APJ:2019}), we have that 
\begin{align}
\label{eq:auxjumps}
\begin{aligned}
\gamma^{i,e}v_i - \gamma^i \widetilde{u}  = \Lambda_i, \quad
\gamma^{i,e} \widetilde{u} + \gamma^i v_i = \Lambda_i.
\end{aligned}
\end{align}
Since $\widetilde{u} \equiv 0$, we have that 
\begin{align*}
 [\gamma v_i]_{\Gamma_i} = \gamma^{i,e}v_i - \gamma^i v_i = 
\gamma^{i,e}v_i -\gamma^i\widetilde{u} - (\gamma^{i,e}\widetilde{u}+\gamma^i v_i) =0,
\end{align*} 
from where it follows that $\{v_i\}_{i=1}^M$ solves Problem \ref{prob:auxprob}.
Assumption \ref{ass:auxk0} implies that
$v_i \equiv 0$, for all $i$ in $\{1,\hdots,M\}$. Finally,
 \eqref{eq:auxjumps} implies $\bm{\Lambda} \equiv0$ as stated.
\end{proof}

\begin{rmrk}
\label{rem:limitedreg2}
Theorem \ref{thm:welposed} {states} that if all the interfaces are of arbitrary smoothness, the solution $\bm\Lambda$ is also arbitrarily smooth. {This result can be generalized to geometries of} limited regularity by following the ideas presented in Remark \ref{rem:limitedreg}, obtaining a solution which is also of limited regularity.
\end{rmrk} 

%%%%%%%%%%%%%%%%%%%%%%%%%
\section{Spectral {Galerkin} Method} 
\label{sec:specmethod}
%%%%%%%%%%%%%%%%%%%%%%%%%
We now provide a numerical method to approximate
solutions of Problem \ref{prob:BIEsystem}, along with its
corresponding error estimates. We restrict ourselves to cases where the
interfaces $\{\Gamma_i\}_{i=1}^{M}$ are $\mathcal{C}^\infty$-Jordan curves. By Theorem \ref{thm:welposed}, the solution is of {arbitrary smoothness, and a spectral method should converge at a super-algebraic rate}
(cf.~\cite[Chapter 9]{saranen2013periodic} and \cite{HU1995340,PaperB}).

%%%%%%%%%%%%%%%%%%%%%%%%
\subsection{Discrete Spaces}
%%%%%%%%%%%%%%%%%%%%%%%%
Let us define a suitable family of finite dimensional 
subspaces of $\bm{\cV}_{\theta}^{{s}_1,{s}_2} $. From the definition of quasi-periodic
Sobolev spaces, it is natural to consider the following finite dimensional functional spaces over $(0,2\pi)$
\begin{align*}
\widehat{\cE}^N_\theta:= \text{span} \{ \widehat{e}_\theta^{n}(t):=e^{ \imath (n+\theta) t }\; : \; n\in\{-N,\hdots,N\}\}.
\end{align*}
It is clear that $\widehat{\cE}^N\subset\widehat{\cE}^{N+1}$ for all $N\in\IN$ and that $\bigcup_{N\in\IN}\widehat{\cE}^N$ is
dense in $H^s_\theta[0,2\pi]$ for any $s \in \IR$. Denoting
$\bz_i:(0,2\pi) \rightarrow \Gamma_i$ a parametrization of $\Gamma_i$, we define
\begin{gather}
\widetilde{\cE}^N_{\theta,\Gamma_i}:= \text{span} \{ \widetilde{e}^n_{\theta,i}:=\widehat{e}_{\theta}^{n}\circ \bz_i^{-1},\; : \; n\in\{-N,\hdots,N\}\}, \\
{\E^N_{\theta,\Gamma_i}}:= \text{span} \{ {e}^n_{\theta,i}:=\norm{\dot{\bz}_i \circ \bz_i^{-1}}{\IR^2}^{-1}\widetilde{e}^n_{\theta,i}\; : \; n\in\{-N,\hdots,N\}\}.
\end{gather}
We can see that $\widetilde{\cE}^N_{\theta,\Gamma_i}$ is the space spanned by finite Fourier basis parametrized on $\Gamma_i$ and that ${\cE}^N_{\theta,\Gamma_i}$ is constructed from the previous space by dividing the basis by the norm of the tangential vector of the corresponding interface.
As before, it is clear that both $\bigcup_{N\in\IN}\cE^N_{\theta,\Gamma_i}$ and $\bigcup_{N\in\IN}\widetilde{\cE}^N_{\theta,\Gamma_i}$ are dense subspaces of $H^s_\theta(\Gamma_i)$ for $s\in \IR$.
Finally, we define the Cartesian product of discrete spaces
\begin{align*}
\bm{\cE}^N_{\theta,\Gamma_i} :=\widetilde{\cE}^N_{\theta,\Gamma_i} \times \cE^N_{\theta,\Gamma_i},
\end{align*}
whose infinite union on $N$ forms a dense subspace of
$\cV^{s_1,s_2}_{\theta,\Gamma_i}$ for any pair $s_1$, $s_2\in\IR$.
%%%%% %%%%%%%%%%%%%%%
\subsection{Discrete Problem}
%%%%% %%%%%%%%%%%%%%%
We now consider the Galerkin discretization of Problem \ref{prob:BIEsystem} on the finite dimensional product space 
\begin{align*}
\mathbb{E}^{\bm{N}}_\theta := \prod_{i=1}^M \bm{\cE}^{N_i}_{\theta,\Gamma_i} \subset \bm{\cV}_{\theta}^{{s}_1,{s}_2} \quad\text{for }\bm{N}=\{N_i\}_{i=1}^M\subset\IN, \quad {s}_1, {s}_2 \in \IR.
\end{align*}

\begin{prblm}[Discrete BIEs]
\label{prob:discBIE}
Let the parameters $k_0$ and $\{\eta_i\}_{i=1}^M$ satisfy
Assumption \ref{ass:bieassump} and let the interfaces $\{\Gamma_i\}_{i=1}^{M}$ be of class $\mathcal{C}^\infty$. For some $\bm{N}=\{N_i\}_{i=1}^M\subset\IN$, we seek $\bm\Lambda^{\bm N}\in \mathbb{E}^{\bm N}_\theta$ such that
\begin{align}
\label{eq:galerkin}
\p{\bm{\bm{\mathcal{M}}}\bm{\Lambda^N},\bm{\Xi}^{\bm N}}_{\bm{\Gamma}} = \p{\bm{\varrho},\bm{\Xi}^{\bm N}}_{\bm{\Gamma}}, \quad \forall\ \bm{\Xi}^{\bm N} \in  \mathbb{E}^{\bm{N}}_\theta,
\end{align}
where the duality product
\begin{align*}
\p{\bm{\Psi},\bm{\Xi}}_{\bm{\Gamma}}:=\sum\limits_{i=1}^{M}\p{\Psi_i,\Xi_i}_{\Gamma_i}\quad\forall\ \bm{\Psi},\ \bm{\Xi}\in \bm{\cV}_\theta^{{s}_1,{s}_2},
\end{align*}
 denotes the sum of two standard duality pairings in $H^\half_\theta(\Gamma_i)$
and $H^{-\half}_\theta(\Gamma_i)$, and $\bm{\varrho}$ accounts for the right-hand side of Problem \ref{prob:BIEsystem}.
\end{prblm}

Since this is a second-kind BIE, we can deduce a quasi-optimality approximation result for the Galerkin
discretization (cf.~\cite[Theorem 4.2.9]{Sauter:2011}), i.e.~there exists
${\bm{N}}^\star = \{{N}^\star_i\}_{i=1}^M$ such that for all ${\bm N}=\{N_i\}_{i=1}^M$ such that $N_i > {N}^\star_i$ for all $i\in\{1,\hdots,M\}$, it holds that 
\begin{align}
\label{eq:quasiopti}
 \norm{\bm{\Lambda} - \bm{\Lambda^N}}{\bm{\cV}^{{s}_1,{s}_2}_\theta} \lesssim 
\inf_{\bm{\Xi}^{\bm N} \in \mathbb{E}^{\bm{N}}_\theta} \norm{\bm{\Lambda} - \bm{\Xi}^{\bm N} }{\bm{\cV}^{{s}_1,{s}_2}_\theta}.
\end{align}
From \eqref{eq:quasiopti} we see that, in order to establish error convergence rates
for the discrete solution, we need to bound those of the
best approximation. From the definition of our discrete and continuous spaces, the problem of bounding the best approximation on $\bm{\cV}^{{s}_1,{s}_2}_\theta$ is equivalent to that of establishing bounds for the best approximation of an
element of $H^s[0,2\pi]$ when approximated by elements of
$\widehat{\cE}^N_{\widetilde\theta}$ with $\widetilde\theta=0$. This issue was already addressed, for example, in \cite[Theorem 8.2.1]{saranen2013periodic}. Specifically,
for any pair $r_1$, $r_2\in\IR$ with $r_2>r_1$ and $f\in H^{r_2}[0,2\pi]$, there holds
\begin{align}
\label{eq:err1}
\inf_{q \in \widehat{\cE}^N_{\widetilde{\theta}}}\|f -q\|_{H^{r_1}[0,2\pi]}  \lesssim  N^{r_1-r_2}\|f\|_{H^{r_2}[0,2\pi]}.
\end{align}

\begin{thrm}
\label{thm:estimate}
Let the parameters $k_0$ and $\{\eta_i\}_{i=1}^M$ satisfy
Assumption \ref{ass:bieassump} and let the interfaces $\{\Gamma_i\}_{i=1}^{M}$ be of class $\mathcal{C}^\infty$. Further assume Assumptions \ref{ass:k0} and \ref{ass:auxk0} to be satisfied and let $s \geq 0$, ${s}_1=s+\half$ and ${s}_2 = s -\half$.  Then, there exists
${\bm{N}}^\star = \{N^\star_i\}_{i=1}^M\subset\IN$ such that for any $\bm{N}= \{N_i\}_{i=1}^M\subset\IN$ with $N_i>N^\star_i$ for all $i\in\{1,\hdots,M\}$, it holds
\begin{align*}
 \norm{\bm{\Lambda} - \bm{\Lambda^N}}{\bm{\cV}^{\half,-\half}_\theta} \lesssim 
 \left({\max_{i\in\{1,\hdots,M\}}} N_i^{-s}\right)\norm{\bm \varrho}{\bm{\cV}^{{s}_1,{ s}_2}_\theta},
 \end{align*}
 where $\bm{\Lambda}$ and $\bm{\Lambda}^{\bm N}$ are the solutions to Problems
 \ref{prob:BIEsystem} and \ref{prob:discBIE}, respectively.
\end{thrm} 

\begin{proof}
For any $\bm{\Xi}^{\bm{N}} \in \mathbb{E}^{\bm{N}}_\theta$, we denote ${\Xi}^{N_i}_i = (\xi^{N_i}_{i}, \zeta^{N_i}_{i})^t$
for all $i\in\{1,\hdots,M\}$, so that 
\begin{align*}
\norm{\bm{\Lambda} - \bm{\Xi}^{\bm{N}}}{\bm{\cV}^{\half,-\half}_\theta}^2 = 
\sum_{i=1}^M  \norm{\lambda_i -\xi^{N_i}_{i}}{H^{\half}_\theta(\Gamma_i)}^2
+\norm{\mu_i -\zeta^{N_i}_{i}}{H^{-\half}_\theta(\Gamma_i)}^2.
\end{align*}
By definition of our continuous and discrete spaces together with
\eqref{eq:err1}, we see that for all $i\in\{1,\hdots,M\}$, one deduces
\begin{align*}
\norm{\lambda_i -\xi_i^{N_i} }{H^{\half}_\theta(\Gamma_i)}^2\lesssim N_i^{-2s} \norm{\lambda_i}{H^{s+\half}_\theta(\Gamma_i)}^2,\quad
\norm{\mu_i -\zeta_i^{N_i} }{H^{-\half}_\theta(\Gamma_i)}^2 \lesssim N_i^{-2s} \norm{\mu_i}{H^{s-\half}_\theta(\Gamma_i)}^2,
\end{align*}
 where the unspecified constant depends only on $\Gamma_i$. Hence, 
 \begin{align*}
\norm{\bm{\Lambda} - \bm{\Xi}^{\bm{N}} }{\bm{\cV}^{\half,-\half}_\theta}^2 \lesssim 
\left({\max_{i\in\{1,\hdots,M\}}} N_i^{{-}2s}\right)\norm{\bm{\Lambda}}{\bm{\cV}^{{s}_1,{s}_2}_\theta}^2.
\end{align*}
Since the problem is well posed, we obtain
\begin{align*}
\norm{\bm{\Lambda} - \bm{\Xi}^{\bm{N}}}{\bm{\cV}^{\half,-\half}_\theta}^2 \lesssim 
\left({\max_{i\in\{1,\hdots,M\}}} N_i^{{-}2s}\right) \norm{\bm{\varrho}}{\bm{\cV}^{{s}_1,{s}_2}_\theta}^2,
\end{align*}
where the unspecified constant now also depends on the wavenumbers $\{k_i\}_{i=0}^{M}$.
\end{proof}

\begin{rmrk}
Theorem \ref{thm:estimate} states that the proposed spectral Galerkin method has a similar performance to the Nyström method, since if interfaces belong to $\mathcal{C}^{\infty}$ then one obtains super-algebraic convergence (commonly observed with the Nyström method \cite{zhang2019fast}). The super-algebraic convergence rate of the Nyström method for the transmission problem on a bounded object in two dimension was rigorously proved in \cite{nystromconvergence}. Similar convergence results for quasi-periodic problems using the
Nyström scheme are, to the best of our knowledge, not available.

%Is interesting to notice that the proof on \cite{nystromconvergence} follows ideas similar to the one that we used (both are heavly inspired by the book of Saranen \cite{saranen2013periodic}), thus it should be possible to obtain the convergence results for Nyström discretizations of quasiperiodic problems by following \cite{nystromconvergence} and using the kernel decomposition presented in Section \ref{ssec:compacteness}, to the best of our knowledge this has not been done. 
\end{rmrk}

\begin{rmrk}
It follows from Remark \ref{rem:limitedreg2} that we can obtain convergence of limited order if the interfaces are of class $\mathcal{C}^{r,1}$ with $r \in [1,\infty)$. 
\end{rmrk}

\subsection{Implementation}
We continue with an overview of the procedure employed to
compute the approximation $\bm{\Lambda^N}$. For a given $N \in \IN$ and $l$, $m\in\IZ$ such that $-N\leq\modulo{l},\modulo{m}\leq N$, integrals 
\begin{align}
\label{eq:integraltypes}
I^1_l := \int_0^{2\pi} f(t) e^{-\imath lt} \;\d\!t\quad\text{and}\quad
I^2_{l,m}:= \int_0^{2\pi} \int_0^{2\pi} F(s,t) e^{-\imath ls} e^{\imath mt}  \;\d\!t \;\d\!s,
\end{align}
where $f$ and $F$ are smooth periodic and bi-periodic functions,
respectively, can be computed to exponential accuracy through the FFT to construct trigonometric
interpolations of the corresponding functions 
(cf.~\cite[Theorem 8.4.1]{saranen2013periodic}). Since the associated kernels
correspond to smooth bi-periodic functions, the computation of the
block matrices $\mathsf{B}_{i,j}$ on \eqref{eq:systemjumps} is performed
in this way.

In terms of computational cost, the set of integrals $\{I^1_l\}_{l=-N}^{N}$ involves $2N+1$ evaluations of the function $f$ and one FFT aplication to a vector of length $2N+1$, whence the total computational cost is $O\left((2N+1) \log(2N+1)\right)$\footnote{This is the classical estimation of the computational cost for the FFT.} arithmetic operations --plus $2N+1$ function evaluations-- to compute the $2N+1$ integrals. For the set of integrals $\{I^2_{l,m}\}_{l,m=-N}^N$, we require $(2N+1)^2$ evaluations of the function $F$, and $2(2N+1)$ FFTs for vectors of length $2N+1$, yielding a
cost of $O\left(2(2N+1) \log(2N+1)\right)$ arithmetic operations (plus $(2N+1)^2$ function evaluations).

On the other hand, the block matrices $\mathsf{A}_i$ in
\eqref{eq:systemjumps} consist of differences of the self-interaction operators on 
$\Gamma_i$ for the four BIOs. While the difference of two
operators is compact---the resulting kernel is smoother than
that associated to a single evaluation of the same operator---the kernel is not
arbitrarily smooth, even if the geometry is. Consequently,
a deeper analysis is required before applying classical algorithms for the computation
of Fourier transforms. 

Let us consider, as an illustrative example, the weakly singular operator. We are required to
compute integrals such as
\begin{align*}
 \int_0^{2\pi} \int_0^{2\pi} \widehat{G}^k_\theta(s,t) e^{-\imath ls} e^{\imath mt} \;\d\!t \;\d\!s,
\end{align*}
where $\widehat{G}^k_\theta$ is as in \eqref{eq:vkernel}. Decomposing $\widehat{G}^k_\theta$ as shown in \eqref{eq:vsplit}, we obtain two integrals,
\begin{align*}
I^S_{l,m} := \int_0^{2\pi}\int_0^{2\pi} S(t-s) J^k_\theta(s,t) e^{-\imath ls} e^{\imath mt} \;\d\!t \;\d\!s,\quad
I^R_{l,m} := \int_0^{2\pi}\int_0^{2\pi} R^k_\theta(s,t) e^{-\imath ls} e^{\imath mt} \;\d\!t \;\d\!s.
 \end{align*}
Since $R^k_\theta(s,t)$ is smooth and periodic (see Section \ref{ssec:compacteness}), $I^R_{l,m}$ may be computed
via the FFT. To compute $I^S_{l,m}$, we use
the expansion (\emph{c.f.}~\cite[Equation 12]{HU1995340}):
\begin{align*}
S(t-s) = \sum_{\substack{n=-\infty \\ n \neq 0}}^\infty \frac{1}{4 \pi n} e^{\imath n(t-s)}.
\end{align*}
Thus, 
 \begin{align}\label{eq:smoothexpIS}
 I^S_{l,m} = \sum_{\substack{n=-\infty \\ n \neq 0}}^\infty \frac{1}{4 \pi n} \int_0^{2\pi} \int_0^{2\pi} J^k_\theta(s,t) e^{-\imath (l+n)s} e^{\imath (m+n)t} \;\d\!t \;\d\!s.
 \end{align}
Since $ J^k_\theta(s,t)$ is smooth and periodic, each of the integrals of
the right-hand side is easy to compute. Moreover, the terms in the series in \eqref{eq:smoothexpIS} decay exponentially and the series may be truncated at the cost of
a small approximation error. Furthermore, the sum in
\eqref{eq:smoothexpIS} may be understood as a discrete convolution, allowing it to be computed by multiplying the corresponding Fourier transforms (see \cite{HU1995340} for details).

The computational cost of computing $\{I^S_{l,m}\}_{l,m=-N}^N$ and $\{I^R_{l,m}\}_{l,m=-N}^N$ is dominated by the latter set of integrals, since it involves $2(N+1)^2$ evaluations of the quasi-periodic Green's function, which is done following \cite{bruno2014rapidly}. The evaluation cost of the quasi-periodic Green's function corresponds to $(2N+1)^2(2N'+1)$ evaluations of the Hankel function, with $N'>N$ is a truncation parameter for the series in \eqref{eq:qpgreenfunc}\footnote{The value of $N'$ has to be chosen depending of $k_0$, but typically one can assume that it need not be greater than $2N$, for $N$ large enough to ensure convergence.}.
Meanwhile, the total cost for $I^S_{l,m}$ is proportional to $(2N+1)\log( 2N+1)$.

For the operators
$\mathsf{K}_\theta^k$ and $\mathsf{K}'^k_\theta$, a similar technique
can be applied using \eqref{dl:splits}. The integrals corresponding to the hyper-singular BIO are approximated
by first using the integration-by-parts formula in Lemma \ref{lemma:intpartsW},
reducing it to two different integrals which are then approximated as those
corresponding to the weakly-singular BIO.
%
%Evaluating smooth kernels is not straightforward as they consist of infinite sums. For their approximation, we rely
%on on the strategy presented in \cite{bruno2014rapidly} for the two-dimensional quasi-periodic Green's function.

Considering $M>1$ interfaces, $2N+1$ degrees of freedom on each interface and $N'$ proportional to $N$ the total cost of the matrix assembly process can be estimated as
$O(N^3M)$ Hankel function evaluations and $O(MN^2\log N)$ arithmetic operations. We point out that the cost could be reduced drastically by constructing an accurate algorithm to approximate the Hankel functions by pre-computing some values. 

\begin{rmrk}
We have restricted ourselves to the analysis of the semi-discrete case, that is, we do not take into consideration the error coming from the approximation of the integrals for the error bound in Theorem \ref{thm:estimate}. However, it is not difficult to incorporate it. Assuming that the parametrizations $\{\bz_i\}_{i=1}^M$ correspond to Jordan curves of class $C^\infty$ and using the aliasing proprieties of the Fourier basis \cite[Chapter 4]{spectralinMatlab} and Lemma \ref{lemma:FourierCofs} we have that the approximation error for the computation of the required integrals is $O(N^{-\ell-1})$, with $2N+1$ being the number of degrees of freedom per interface and $\ell$ an arbitrarily large integer. Then, the fully discrete error can be obtained by an {application of Strang's lemma} \cite[Section 4.2.4]{Sauter:2011}, from where it follows that the behaviour of the fully discrete error, with respect to $N$, is the same as in Theorem \ref{thm:estimate}. 
\end{rmrk}

%%%%%%%%%%%%%%%%%%%%%%%%%
\section{Numerical Examples}
\label{sec:numexam}

%%%%%%%%%%%%%%%%%%%%%%%%%

We now showcase computational experiments {verifying} the convergence estimates found in Theorem \ref{thm:estimate}.
The implementation of the aforementioned
algorithms was carried through a C++ cpu-only library. All the experiments ran on a Intel I7-4770@3.4GHZ {processor} with 8 threads. {The code was compiled with gcc 4.9.4, openmp and O2 flags on.}

\subsection{Code Validation}

{We begin by considering the simple case of a grating with two media separated by a single horizontal line segment acting as its layer}. Hence, using the following expansion of the Green's function \cite[Proposition 4.2]{APJ:2019}:
$$G^k_\theta (\bx,\by) = \frac{\imath}{4 \pi} \sum_{j\in\IZ} \frac{1}{\beta_j}e^{\imath \beta_j \modulo{x_2-y_2}-\imath j_\theta(y_1-x_1)} \text{  for all } \bx, \by \in \IR^2,$$
it is possible to assemble the matrix analytically. The matrix $\bm{\mathcal{M}}$ is then composed of only block diagonal terms. Since the right-hand side only has two non-null components\footnote{One for the Dirichlet trace of the incident wave and another for the Neumann trace.}, {only the corresponding components for the solution are non-zero, yielding a closed form for the solution.}

In order to test the implementation, we consider an artificial (harder) problem by including ghost domains, i.e., we add extra  smooth (ghost) layers that separate domains with the same refraction index. Hence, the solution is the same as if these additional domains did not exist and has a closed form, as before. The results for different ghost layers are reported in Figure \ref{Fig:ConvgGhost}. 
\begin{figure}[t]
  \subfigure[
	Geometry 
  ]{   
    \includegraphics[width=.47\textwidth]{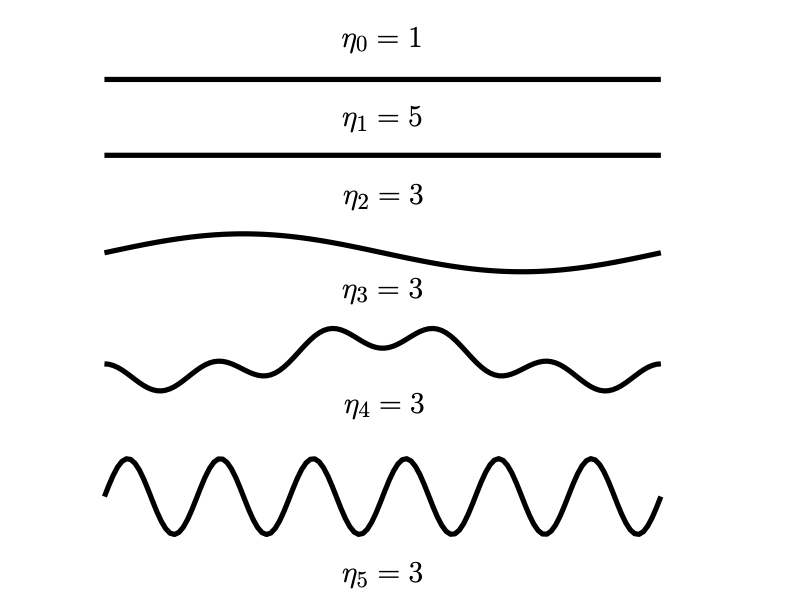}
  }
  \subfigure[
   Convergence Plot
  ]{
  \includegraphics[width=.47\textwidth]{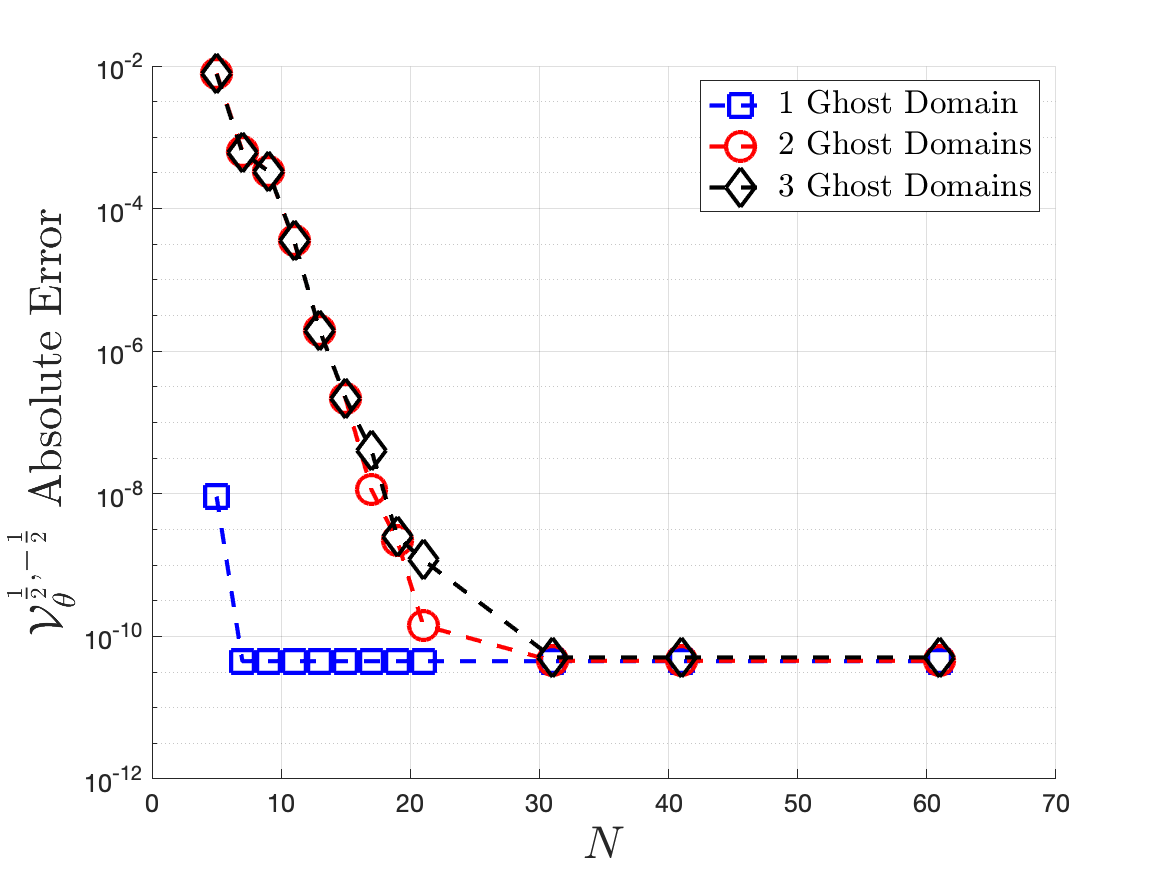}
  }
  \caption{ 
Subfigure (a) shows the problem geometry. Subfigure (b) shows the error in the $\bm{\cV}^{{\half,-\half}}_\theta$ norm with respect to the analytic solution. We have included results for different numbers ghost layers (1,2 and 3, respectively), i.e., the first experiment considers only the first 3 layers (counting downwards), the second one considers the first 4 layers and the third considers all 5 layers.
}
\label{Fig:ConvgGhost}
\end{figure}

We also display the convergence behaviour of the method for interfaces with limited regularity by repeating the previous experiment (same incident field) with one ghost domain and an interface given by 
\begin{align*}
\bz_3(t) = (t,a|\sin(t)|^p+b),
\end{align*}
where $a,b$ are real numbers that scale the interface, and $p$ is an odd integer that determines the smoothness {degree} of the interface. In particular, $\bz_3$ is in $\mathcal{C}^{p-2,1}$ or, more precisely, $\mathcal{C}^{p-1}$ with an integrable $p$-th derivative. Results are reported in Figure \ref{Fig:Convgp}. 

For all experiments in this section, the frequency is chosen as $k_0=1$ and the incidence angle is $0.47$ radians.

\begin{figure}[b]
\centering
      \includegraphics[width=.55\textwidth]{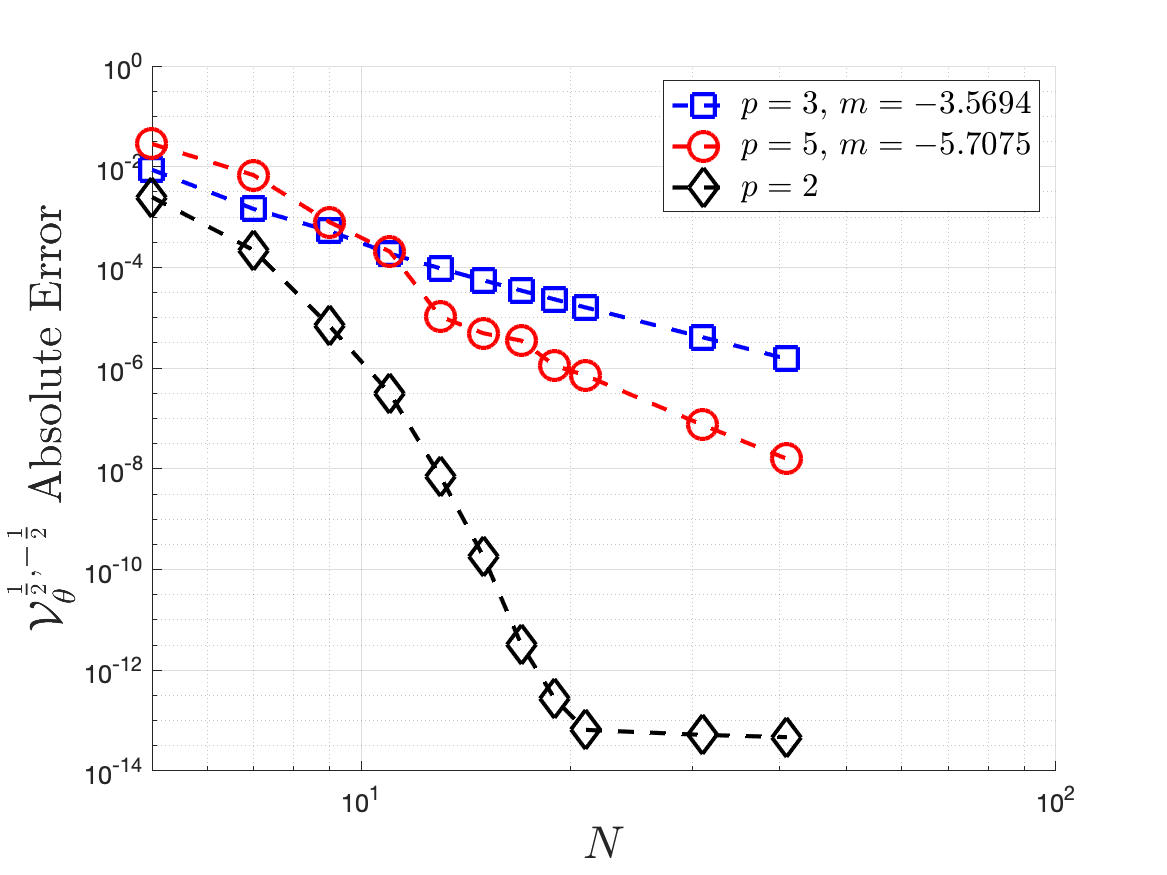}   
  \caption{ 
Error in the $\bm{\cV}^{{\half,-\half}}_\theta$ norm with respect to the analytic solution. The legend indicates an estimate of the slope of the error convergence curves for different values of $p$ (degrees of smoothness). Classically, error convergence estimates for spectral methods indicate the slope to be at least equal to $p$. We also consider the case $p=2$, where the extra layer is $\mathcal{C}^\infty$ and the super-algebraic convergence rate is observed. 
}
\label{Fig:Convgp}
\end{figure}

%%%%%%%%%%%%%%%%%%%%%%%%%%%%%%%%%%%%%%%%
\subsection{Convergence results}
%%%%%%%%%%%%%%%%%%%%%%%%%%%%%%%%%%%%%%%%
We now consider a smooth geometry composed of the 12 layers
and varying refraction indices. Two different scenarios for the choice of indices are employed, reported in Table \ref{tb:Ex1} ($\eta_i^{(1)}$ and $\eta_i^{(2)}$ for the first and second cases, respectively). We also consider three different wavenumbers for the incident wave, $k_0 = 2.8,$ $14$ and $28$. 
Convergence
results in the energy norm for the solution of Problem \ref{prob:discBIE} for the different cases of parameters and wavenumbers are reported in Figure \ref{Fig:Convg}, where exponential convergence is observed for all considered scenarios, as expected. All errors were computed with respect to an overkill solution, with approximately 50 more bases per interface than the last plotted point for each curve. The incidence angle is, again, $0.47$ radians.
%\jp{While is not an easy task to establish the number of basis needed, for a desire accuracy, in terms of the frequency $k_0$ and the refraction indices $\eta_i$, it seems from the presented experiments that they should be chosed proportional to the maximum value $k_i = \eta_i k_0$.} 

\begin{table}%[H]
\centering
\begin{tabular}{l|llllllllllll}
         & 1    & 2    & 3    & 4    & 5    & 6    & 7    & 8    & 9    & 10   & 11   & 12   \\ \hline
$\eta^{(1)}_i$ & 4.7 & 4.2 & 4.8 & 3.6 & 1.1 & 4.4 & 4.7 & 3.7 & 4.0 & 3.9 & 2.6 & 3.6
\\ 
$\eta^{(2)}_i$ & 4.7 & 8.4 & 4.8 & 7.2 & 1.1 & 8.8 & 4.7 & 7.4 & 4.0 & 7.8 & 2.6 & 7.2
\end{tabular}
\caption{Value of the refraction indices $\{\eta^{(1)}_i\}_{i=1}^{12}$
and  $\{\eta^{(2)}_i\}_{i=1}^{12}$ (corresponding to the two considered cases) for the grating
in Figure \ref{fig:PlotEx1} (counting downwards). }
\label{tb:Ex1}
\end{table}

\begin{figure}[ht!]
  \subfigure[
	Geometry 
  ]{   
    \includegraphics[width=.47\textwidth]{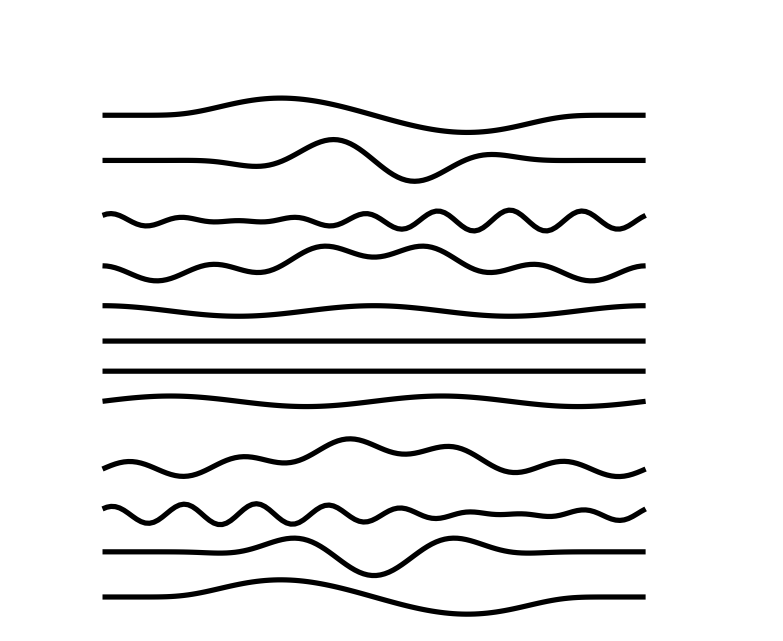}
  }\hspace{0.2cm}
  \subfigure[
   $k_0 = 2.8$
  ]{
  \includegraphics[width=.47\textwidth]{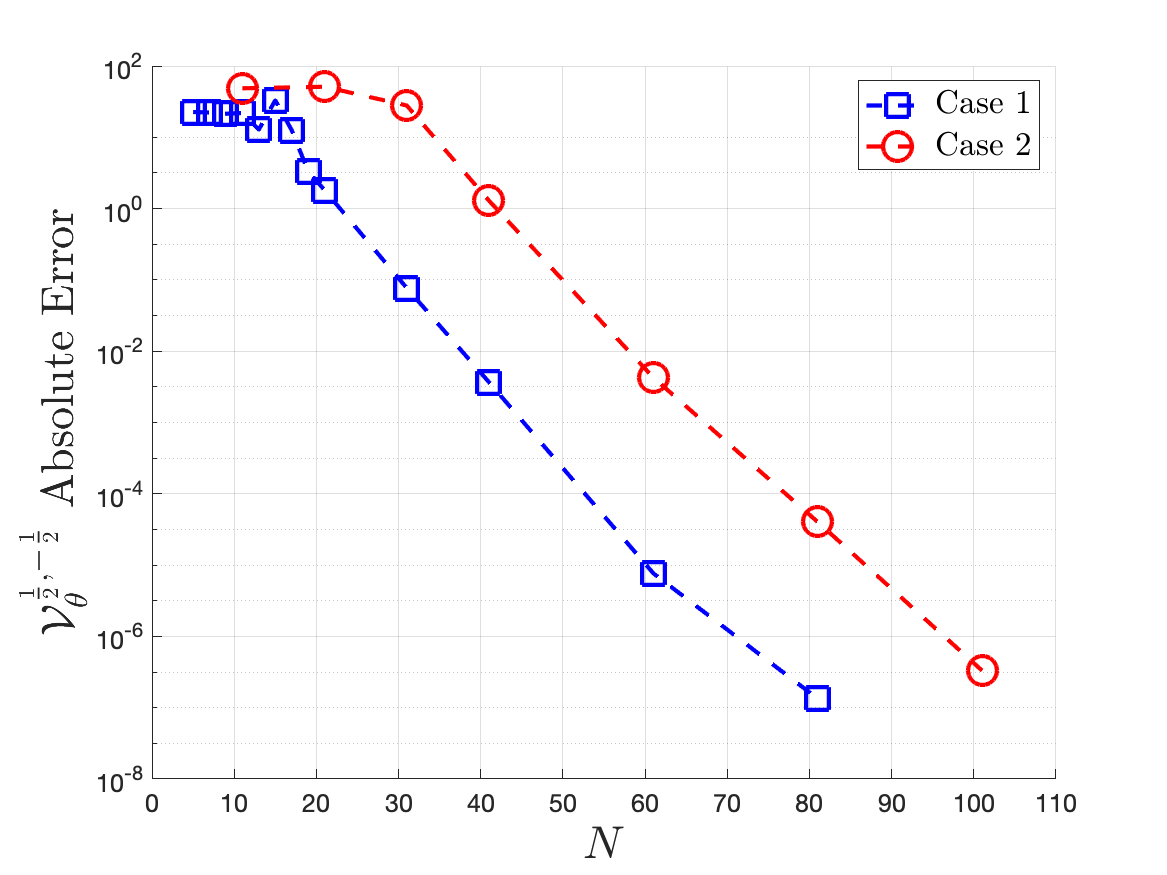}
  }
\subfigure[
   $k_0 = 14$
  ]{
  \includegraphics[width=.47\textwidth]{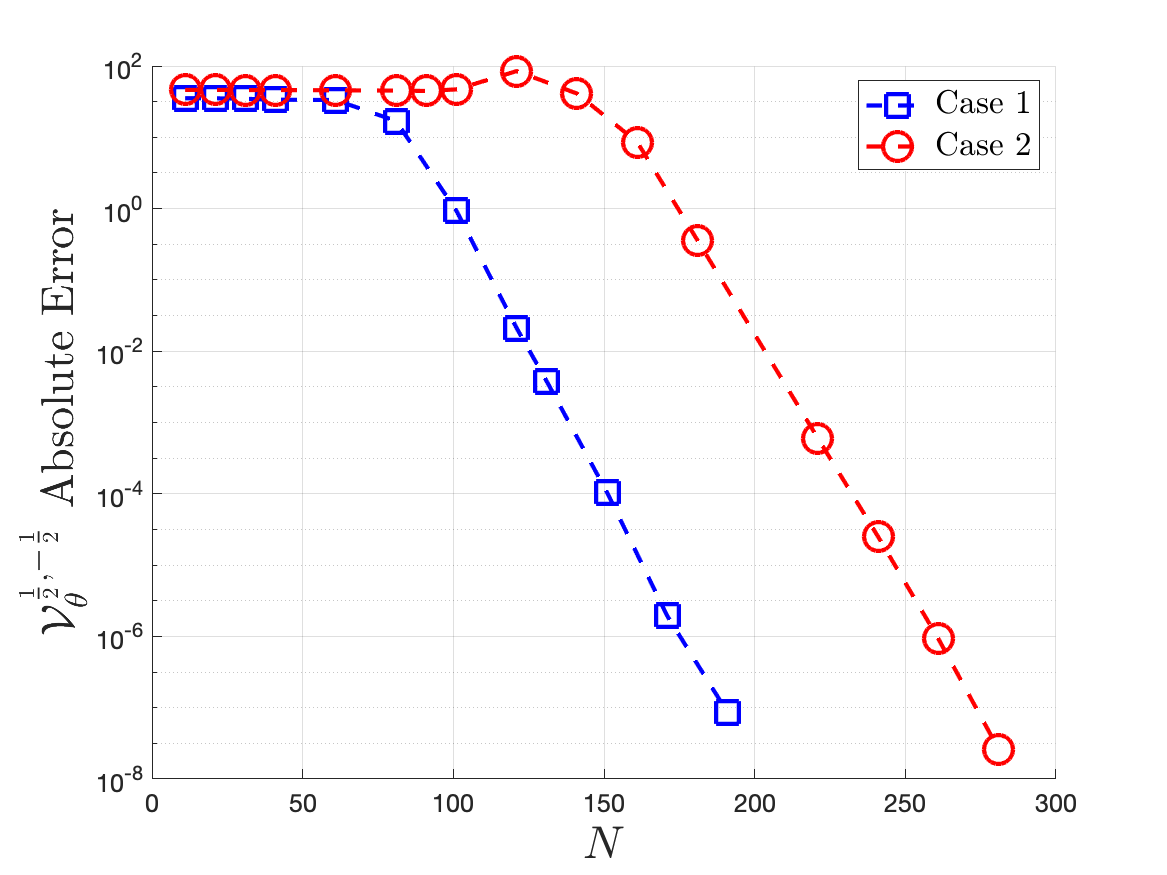}
  }\hspace{0.2cm}
\subfigure[
   $k_0 = 28$
  ]{
  \includegraphics[width=.47\textwidth]{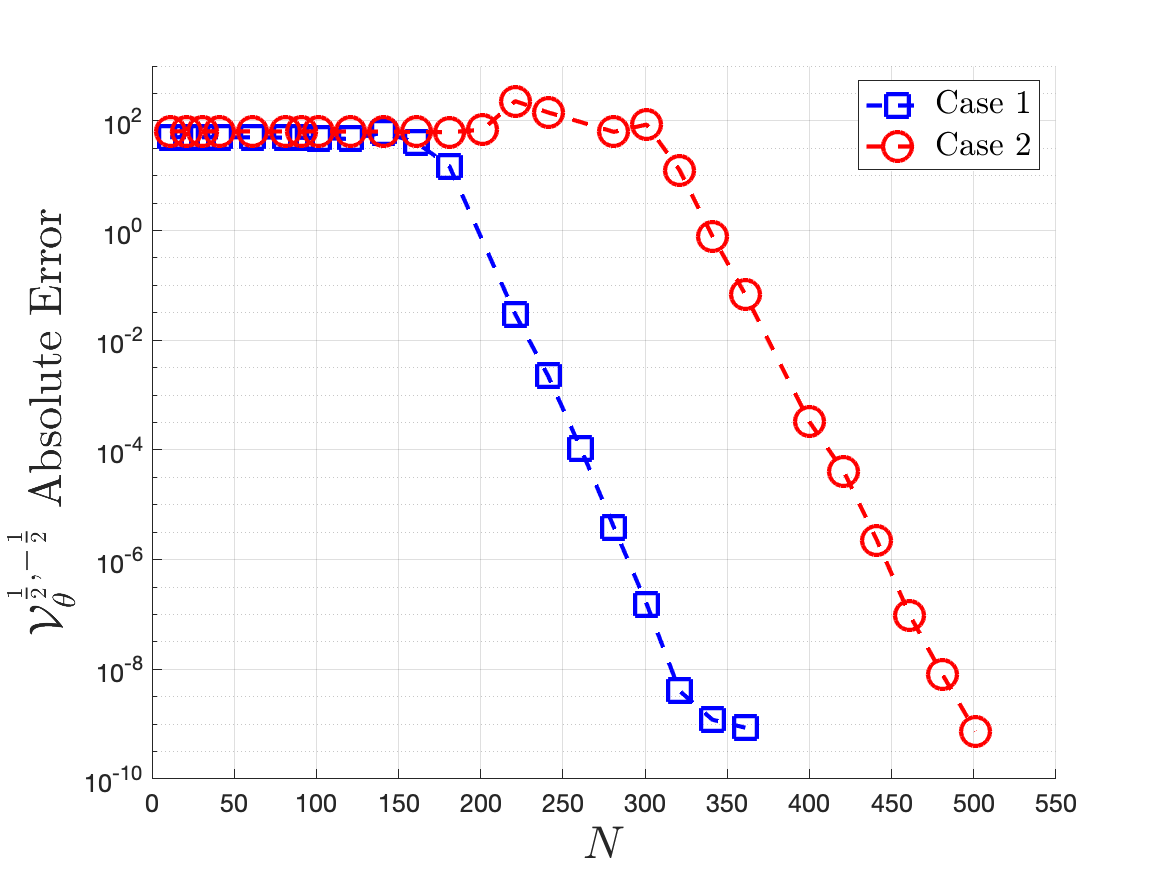}
  }
  \caption{
Subfigure (a) shows the problem geometry (with 12 layers). Subfigures (b), (c) and (d) display the errors (in the corresponding energy norm) for the different values of $k_0$, i.e., $2.8$, $14$ and $28$, respectively. Each of these subfigures present error convergence curves for the two scenarios of refraction indices considered and specified in Table \ref{tb:Ex1}. Notice that the 
curves in red---corresponding to parameters $\eta_i^{(2)}$ in Table \ref{tb:Ex1}---display a longer preasymptotic regime before convergence is observed for all considered values of $k_0$, seemingly due the presence of layers with higher wavenumbers (see Remark \ref{rmk:conv_N}).}

\label{Fig:Convg}
\end{figure}
Finally, in Figure \ref{fig:PlotEx1} we present an illustration of the total field corresponding to the refraction indices given in Table \ref{tb:Ex1}.

\begin{figure}[ht!]
\centering
\subfigure[$k_0=2.8$.]{
\includegraphics[width=.3\linewidth]{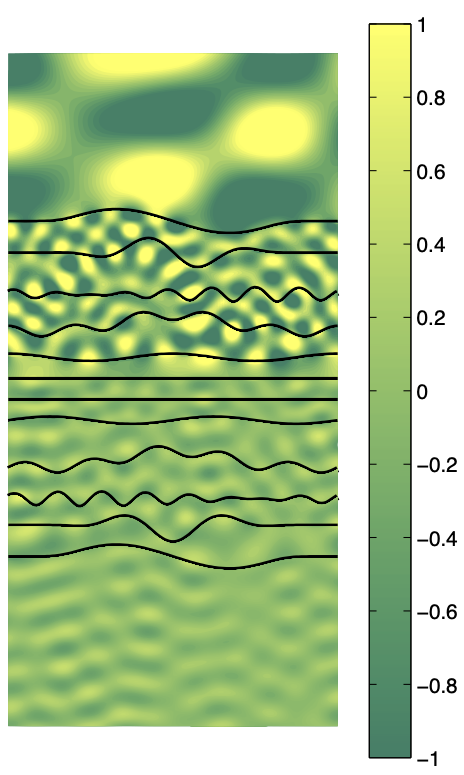}
\label{subfig:errP1}
}
\hspace{1cm}
\subfigure[$k_0=14$.]{
\includegraphics[width=.3\linewidth]{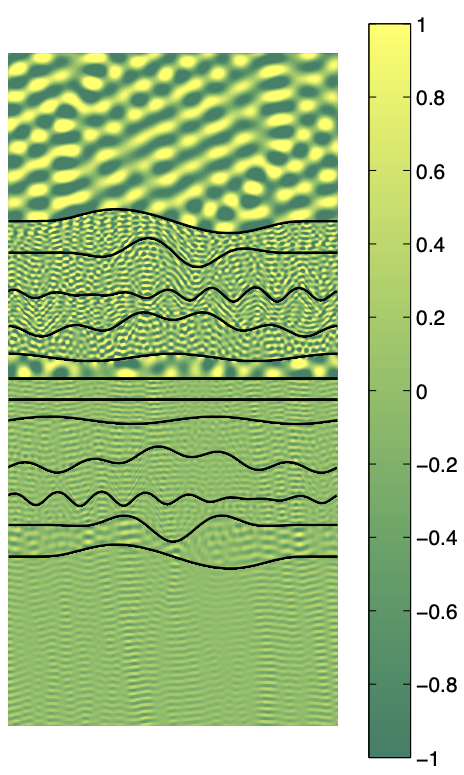}
\label{subfig:errP2}
}
\hspace{1cm}
\subfigure[$k_0=28$.]{
\includegraphics[width=.3\linewidth]{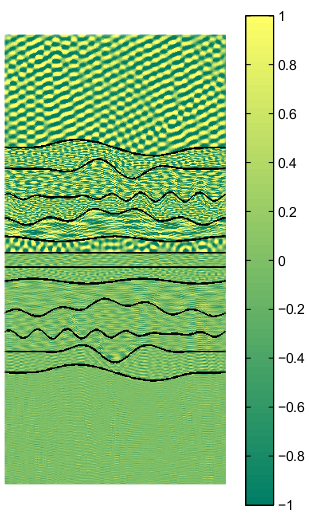}
\label{subfig:errP3}
}
\caption{Real part of the total wave ($u^{\text{(tot)}}=u^{\text{(sc)}}+u^{\text{(inc)}}$) for
each different value of $k_0$, namely $2.8$, $14$ and $28$.
The refraction indices on each layer are those indicated on Table
\ref{tb:Ex1}. The incidence angle is again $0.47$. }
\label{fig:PlotEx1}
\end{figure}
\begin{rmrk}
\label{rmk:conv_N}
Though establishing the relation between the parameters---$\{\eta_{i}\}_{i=0}^{M}$ and $k_0$---and the number of basis elements required to attain a certain desired accuracy is not straightforward, our experiments suggest that $N$ should be chosen
proportional to the maximum wavenumber $k_{\text{max}}:=\max_{i\in\{0,\hdots M\}}k_i$.
\end{rmrk}
%%%%%%%%%%%%%%%%%%%%%%%%%
\section{Conclusions}
\label{sec:conclusions}
%%%%%%%%%%%%%%%%%%%%%%%%%
We have proposed a fast spectral method for the efficient representation, through surface potentials based on the quasi-periodic Green's function,
for the solution of the Helmholtz equation with transmission boundary conditions on a periodic domain. Theorem \ref{thm:estimate}, we obtained convergence estimates for 
the discrete approximation of the corresponding boundary data,
and found that discrete solution converge at a super-algebraic rate to continuous solutions of the considered boundary integral equation. Though, we focused
on the Helmholtz transmission problem, our approximation
results and convergence estimates can be easily extended to
other boundary integral equations on quasi-periodic Sobolev spaces whenever
the formulation is well posed.
We avoided Rayleigh-Wood anomalies
from our analysis since the series in \eqref{eq:qpgreenfunc} fails
to converge for said frequencies and, for the same reason, our previous results from \cite{APJ:2019} exclude them as well. 

Though similar numerical results are known for the Nyström Method, theoretical results confirming the observed convergence
rates are scarce \cite{nystromconvergence}, which is an advantage of Galerkin discretizations
such as that presented in this article. Moreover, the convergence rate for the proposed discretization is equal to that expected of Nyström methods, so that it is numerically competitive with them 
while inheriting the theoretical benefits of a Galerkin discretization. 

Future work considers: (i) including Rayleigh-Wood anomalies to our analysis, (ii) extending our results to three
dimensional Helmholtz equations and Maxwell's equations on periodic domains and (iii) applications in uncertainty quantification \cite{silva2018quantifying} and shape optimization \cite{Aylwin2020}.

\bibliographystyle{acm}
\bibliography{max}

\end{document}

%% file: ex_shared.tex
\usepackage{bm}
\usepackage{psfrag}
\usepackage{xcolor}
\usepackage[all,cmtip,line]{xy}
\usepackage[normalem]{ulem}
\usepackage{tikz} 
\usepackage[utf8]{inputenc}
\usepackage{pgfplots}
\usepackage{scalerel}
\usepackage{amssymb}
\usepackage{caption}
\usepackage{amsmath}
\usepackage{mathabx}
\usepackage{subfigure}
\usepackage[shortlabels]{enumitem}
\usepackage{relsize}
\usepackage{cleveref}
\usepackage{graphicx}
\usepackage[T3,T1]{fontenc}
\usepackage{amsthm}

\usepackage{cases}% http://ctan.org/pkg/cases

\ifpdf
  \DeclareGraphicsExtensions{.eps,.pdf,.png,.jpg}
\else
  \DeclareGraphicsExtensions{.eps}
\fi

\DeclareSymbolFont{tipa}{T3}{cmr}{m}{n}
\DeclareMathAccent{\invbreve}{\mathalpha}{tipa}{16}

\newtheorem{thrm}{Theorem}[section]
\newtheorem{lmm}[thrm]{Lemma}
\newtheorem{crllr}[thrm]{Corollary}
\newtheorem{prpstn}[thrm]{Proposition}

\newtheorem{assumption}[thrm]{Assumption}
\newtheorem{dfntn}[thrm]{Definition}

\newtheorem{prblm}[thrm]{Problem}
\newtheorem{rmrk}[thrm]{Remark}

\newcommand{\p}[1]{\left\langle #1\right\rangle}
\newcommand{\pr}[1]{\left( #1\right)}

\newcommand{\modulo}[1]{\left\vert #1\right\vert}

\colorlet{MAGENTA}{magenta}

\definecolor{forestgreen}{rgb}{0.13, 0.55, 0.13}

\newcommand{\be}{\begin{equation}}
\newcommand{\ee}{\end{equation}}
\newcommand{\C}{\mathcal C}
\newcommand{\cD}{\mathcal D}
\newcommand{\E}{\mathcal E}
\newcommand{\cE}{\mathcal E}

\newcommand{\cO}{\mathcal O}

\newcommand{\cV}{\mathcal V}

\newcommand{\cG}{\mathcal G}

\newcommand{\V}{{\mathsf V}}
\newcommand{\W}{{\mathsf W}}
\newcommand{\Kk}{{\mathsf K}}

\newcommand{\IC}{{\mathbb C}}

\newcommand{\IN}{{\mathbb N}}

\newcommand{\IR}{{\mathbb R}}

\newcommand{\IZ}{{\mathbb Z}}
\newcommand{\norm}[2]{\left\lVert #1\right\rVert_{#2}}

\newcommand{\Id}{\operatorname{\mathsf{I}}}

\renewcommand{\d}{\!\!\operatorname{d}}
\newcommand{\fpcs}[1]{\operatorname{\mathsf{f}\pc}}
\newcommand{\pc}{}%}}

\newcommand{\scurl}{\operatorname{curl}}

\newcommand{\bn}{\bm{n}}

\newcommand{\bz}{\bm{z}}

\newcommand{\bx}{\bm{x}}
\newcommand{\by}{\bm{y}}

\newcommand{\half}{\frac{1}{2}}

\newcommand{\tD}{\gamma_{\mathrm{D}}}
\newcommand{\tN}{\gamma_{\mathrm{N}}}

\def\Xint#1{\mathchoice
   {\XXint\displaystyle\textstyle{#1}}%
   {\XXint\textstyle\scriptstyle{#1}}%
   {\XXint\scriptstyle\scriptscriptstyle{#1}}%
   {\XXint\scriptscriptstyle\scriptscriptstyle{#1}}%
   \!\int}
\def\XXint#1#2#3{{\setbox0=\hbox{$#1{#2#3}{\int}$}
     \vcenter{\hbox{$#2#3$}}\kern-.5\wd0}}

\def\dashint{\Xint-}